\documentclass[11pt,reqno, a4paper]{amsart}

\usepackage{amssymb}

\usepackage[english]{babel}

\usepackage[dvipsnames,svgnames,x11names]{xcolor}
\usepackage[hyphens]{url}
\usepackage[colorlinks=true,linkcolor=Maroon,citecolor=blue,urlcolor=blue,hypertexnames=false,linktocpage]{hyperref}
\usepackage{bookmark}
\usepackage{amsmath,thmtools,mathtools}
\mathtoolsset{showonlyrefs=true}
\usepackage{bm}
\usepackage{mathtools}

\usepackage{fancyhdr}
\usepackage{esint}
\usepackage{enumerate}

\usepackage[scr]{rsfso}
\usepackage{mathtools}
\usepackage{pictexwd,dcpic}
\usepackage{graphicx}

\def\avint{\mathop{\mathchoice{\,\rlap{-}\!\!\int}
		{\rlap{\raise.15em{\scriptstyle -}}\kern-.2em\int}
		{\rlap{\raise.09em{\scriptscriptstyle -}}\!\int}
		{\rlap{-}\!\int}}\nolimits}

\usepackage[labelfont=bf, font=small]{caption}

\newcounter{mgncount}

\declaretheorem[name=Theorem,numberwithin=section]{thm}

\declaretheorem[name=Lemma,sibling=thm]{lemma}
\declaretheorem[name=Proposition,sibling=thm]{prop}

\declaretheorem[name=Corollary,sibling=thm]{cor}

\declaretheorem[name=Remark,style=remark,sibling=thm]{bem}

\numberwithin{equation}{section}

\newcommand{\cM}{\mathcal{M}}

\newcommand{\mc}{\mathcal}

\usepackage{dsfont}
\newcommand{\R}{\mathds{R}} 
\newcommand{\KGauss}{\mathscr{K}}
\newcommand{\heatoperatorDOSdim}{ \left( \partial_t - \partial_s^2 \right)}
\newcommand{\pr}{ \partial_r} 
\newcommand{\pu}{ \partial_u} 
\newcommand{\ps}{ \partial_s} 
\newcommand{\kcirc}{\overset{\circ}{\kappa}}

\newcommand{\avg}[1]{\langle #1 \rangle}
\def\restrict#1{\raise-.5ex\hbox{\ensuremath|}_{#1}}

\usepackage{scalerel}[2014/03/10]
 \usepackage[usestackEOL]{stackengine}
\def\intavg{\,\ThisStyle{\ensurestackMath{%
    \stackinset{c}{0\LMpt}{c}{0\LMpt}{\SavedStyle-}{\SavedStyle\phantom{\int}}}%
    \setbox0=\hbox{$\SavedStyle\int\,$}\kern-\wd0}\int}

\usepackage{graphicx}
\def\tang{\ThisStyle{\abovebaseline[0pt]{\scalebox{-1}{$\SavedStyle\perp$}}}} 

\numberwithin{equation}{section}

\usepackage{xcolor}
\definecolor{color-cites}{HTML}{9a2144}
\definecolor{brickred}{HTML}{e63d12}
\definecolor{royalblue}{HTML}{2054e3}
\definecolor{blavet}{HTML}{3adbcc}
\definecolor{verd}{HTML}{4a6741}

\usepackage{tikz}

\begin{document}
	
	\title[CCF on pinched Hadamard  surfaces]{Constrained curvature flows \\ on  pinched Hadamard surfaces}
	\author{Sara Albert-Niclòs, Esther Cabezas-Rivas,}

	\address{S. Albert Niclòs: : Departament de Matem\`atiques,
		Universitat de Val\`encia, Av. Vicent Andrés Estellés 19, 46100 Burjassot, Spain.
		{\tt sara.albert@uv.es  }}
	
	\address{E. Cabezas-Rivas: Departament de Matem\`atiques,
		Universitat de Val\`encia, Av. Vicent Andrés Estellés 19, 46100 Burjassot, Spain.
		{\tt esther.cabezas-rivas@uv.es  }}

\date{\today. The authors have been partially supported by project PID2022-136589NB-I00  funded by
	MCIN/AEI/10.13039/501100011033 and by ERDF A way of making Europe. The first author has an additional predoctoral grant PREP2022-000011 tied to the aformentioned project. Both authors ave also been partially supported by the
	project CIAICO/2023/035 funded by Conselleria d’Educació, Cultura, Universitats i Ocupació.}	
	
	\begin{abstract}
We study area- and length-preserving curvature flows for embedded closed curves on pinched Hadamard surfaces. In the variable-curvature setting, the evolution equations contain additional lower-order terms, so the PDE analysis requires refined comparison arguments and delicate curvature estimates. For smooth convex initial curves, we prove preservation and instantaneous strictness of convexity, long-time existence, and uniform bounds for the curvature and its higher derivatives. Under additional geometric assumptions, we obtain convergence of the curvature to a constant. In the rotationally symmetric case, the area-preserving flow exhibits a dichotomy: either the evolving curves converge exponentially to a geodesic circle, or they drift off to infinity and approach a constant-curvature limit curve. We also identify a geometric condition on the initial curve that prevents escape to infinity and guarantees convergence to a geodesic circle.
	\end{abstract}
	\keywords{Constrained curvature flow; Hadamard pinched surfaces; Rotationally symmetric; area-preserving; length-preserving}
	 \subjclass[2020]{53E10, 35R01, 35K59, 35K93, 35B35} 
	\maketitle

	\section{Introduction and statement of main results}


Let $\gamma_0: \mathbb{S}^1 \to \cM$ be an embedded, smooth and closed  curve, where $(\cM,g)$ is a complete and oriented surface, and let $\mathcal{T}>0$. The family $\gamma: \mathbb{S}^1 \times [0, \mathcal T) \to \cM$ evolves under the \textit{constrained curvature flow} (CCF, for short) if it satisfies
 \begin{equation}
\left\{
\begin{array}{l}
\partial_t \gamma(p,t) =  \big( h(t) - \kappa(p,t) \big) N(p,t), \smallskip \\
\gamma(p,0) = \gamma_0(p).
\end{array}
\right.
\label{eq:ccf_kappa}
\end{equation}
where $\kappa$ denotes the (intrinsic or geodesic) curvature of $\gamma$, and $N$ is the unit outer normal vector to the evolving  curve $\Gamma_t = \gamma(\mathbb S^1,t)$.  
The sign convention is chosen so that for convex curves (i.e., those with $\kappa > 0$) the curvature vector points inwards.  
Here $h(t)$ is a global term depending on the chosen constraint:
 \begin{equation}
h(t)=\frac{\int_{\mathbb{S}^1} \kappa^{1+\alpha} \ ds}{\int_{\mathbb{S}^1} \kappa^\alpha \ ds};
\label{eq:def_h_LPCF_f}
\end{equation}
 $h$ corresponds for $\alpha = 0$ to the average curvature and the flow is called the \emph{Area-Preserving Curve Shortening Flow} ({\sc ap-csf}), while for $\alpha =1$ one gets the \emph{Length-Preserving Curvature Flow} ({\sc lp-cf}). Both represent non-local constrained versions of the  curve shortening flow.

As the name suggests,  in the {\sc ap-csf} the presence of the global term keeps the enclosed area constant throughout the evolution  while the length is decreasing with time. This makes the flow especially interesting for isoperimetric applications, but also more delicate to analyze since standard tools such as comparison principles may fail and, consequently, e.g. embedded curves can develop self-intersections (cf.~\cite{MS00}).


In this work, we consider curves evolving on 
\emph{pinched Hadamard surfaces}, that is, complete, simply connected Riemannian surfaces that generalize the hyperbolic plane by allowing the Gauss curvature $\KGauss$ to vary between two negative bounds
\begin{equation} \label{pinched}
-b^2 \leq \KGauss \leq -a^2, \qquad 0 < a < b.
\end{equation}
We also obtain some results in the case of bounded (non necessarily negative) Gauss curvature. Let us highlight that most of the previous literature assume constant  ambient curvature. In fact, compared with the flat or constant-curvature case, the variable ambient geometry introduces additional lower-order terms into the evolution equations, making the a priori PDE estimates substantially more delicate and preventing a direct use of standard maximum-principle arguments. The analysis therefore requires a careful combination of comparison geometry, refined evolution identities, and sharp curvature estimates.

\smallskip

 Here we manage to deal with all the new technical challenges appearing because of the variation in $\KGauss$, showing there are phenomena during the evolution that do not appear in the constant curvature case. More precisely, we prove

\begin{thm} \label{main}
	Let $\cM$ be a pinched Hadamard surface, and let $\gamma_0: \mathbb{S}^1 \to \cM$ be an embedded, smooth, closed and convex  curve. Then the solution $\gamma:\mathbb{S}^1\times [0,\mathcal{T})\to \cM$ of \eqref{eq:ccf_kappa} with global term $h(t)$ defined as in \eqref{eq:def_h_LPCF_f} and initial condition $\gamma_0$ satisfies the following:
	\begin{enumerate}
	\item[{\rm (1)}]   $\gamma_t = \gamma(\cdot, t)$ is strictly convex for all $t > 0$.
	\smallskip
	\item[{\rm (2)}] $\mathcal T =\infty$, that is, the solution exists for all time $t \in [0, \infty)$. 
	\smallskip
	\item[{\rm (3)}] If $\sup_{\Gamma_t}|\nabla^\ell \KGauss| \leq {C_2}$ for $\ell =1,2$, then the curvature of $\gamma_t$ subconverges uniformly to a constant.
	\smallskip
	\item[{\rm (4)}]  For the {\sc ap-csf}, assume that $\cM$ is rotationally symmetric with $\KGauss$ decreasing with respect to the radial distance to the pole. Then
	exactly one of the following two alternatives holds:
	\begin{enumerate}
		\smallskip
		\item the solution converges exponentially fast in the $C^\infty$-topology to a geodesic circle centered at the pole enclosing the same area as $\gamma_0$, or,
		\smallskip
		\item there is a subsequence that ``diverges'' to spatial infinity, being the limiting shape a constant curvature curve not centered at the pole. 
	\end{enumerate}
\end{enumerate}
\end{thm}

Based on the previous dichotomy, we manage to find an additional geometric condition on the initial curve $\gamma_0$ ensuring that the evolving convex curve cannot escape to infinity (cf. Proposition \ref{lemma:condicio_area_distancia_radial_fitada}). The latter is a distinctive geometric behaviour that cannot occur in ambient spaces of constant curvature. The preservation of convexity in (1) actually happens for a more general class of fully non-linear constrained curvature flows (see Lemma \ref{lema:preservationofstrictconvexity}), while the uniform subconvergence of the curvature to a constant also holds in a more general setting, as it needs bounded (non-necessarily negative) Gauss curvature (for concrete details see Theorem \ref{teo:limit_kappa_es_h}). In general, there are many results that are valid for any constrained flow i.e. for any $\alpha$ in the definition of the global term $h(t)$.
\smallskip

On general Riemannian surfaces, the geometry of curves with constant curvature can be quite complex, especially when $\KGauss$ is not constant. Nevertheless, convergence of the curvature function $\kappa$ to a spatially constant value indicates that the flow is regularizing the shape of the curve, making it increasingly uniform in a geometric sense. The extra hypotheses in (4) are precisely imposed to guarantee that there is a unique candidate for a constant curvature curve.

\smallskip

Broadly speaking, a key difficulty in the present work is that the ambient geometry is non-flat and the Gauss curvature is not constant, so the evolution equations contain extra lower-order terms that are absent in the model spaces. This makes the PDE estimates considerably more intricate and requires a careful combination of maximum-principle arguments, curvature comparisons, and auxiliary geometric estimates in order to recover the needed a priori bounds.

\smallskip
About the techniques used to prove the above results, we try to follow the general strategy for the volume-preserving flow of hypersurfaces designed by \cite{Huisken_1987_VPMCF} in the flat case, and enhanced by the second author and Miquel in \cite{CabezasRivasMiquel2007}  for the hyperbolic space. However, we actually have to deal
with an ambient space non necessarily flat and with sectional curvatures non necessarily
constant. The key to overcome these further difficulties is the mixture of some ideas from the
aforementioned works  with several new technical ingredients deployed throughout the paper.

\smallskip

It would be impossible to enumerate here all the refined arguments and technical ingredients we developed ad hoc to handle the non-constant ambient curvature; but let us highlight a couple of illustrative examples. In the hyperbolic space, $\sinh r \, \partial_r$, where $r$ is the radial coordinate, is a conformal Killing field; as a consequence, many terms in the evolution of the support function (which is crucial to control the curvature under the flow) vanish or are easily computed. In our case, this is no longer true, and we have to find and deal with a much more complicated evolution equation (see Proposition \ref{eq:heat_operator_u}), all whose extra terms have to be suitably estimated. 

\smallskip
We managed to do so by using the trick of writing some parts in terms of the curvature $\kcirc$ of geodesic circles, which is well behaved, and by establishing a new comparison result for curves on pinched Hadamard surfaces (cf.~Lemma~\ref{lema:hessian_inequalities_r} (c)), which provides precise two-sided bounds for the Hessian of the radial distance function. These bounds are essential for controlling the geometry of the evolving curves and, in particular, allow us to prove that the curvature of the curves remains controlled under the flow.

\smallskip
On the other hand, geodesic circles on Hadamard surfaces do not need to have constant curvature, and hence the standard ideas to show that the evolving curve encloses an inball of fixed center for a uniform amount of time also fail in our setting 
and we have to devise novel arguments to reach the same conclusion. In short, This is the first setting where the standard geometric barriers no longer come for free because geodesic circles themselves do not have constant curvature. Similarly, there are numerous further instances in which we must introduce new ideas, and the fine details of the proofs diverge substantially from the original methods, even though the overall strategy may appear similar at first glance.

 \smallskip
 
 Moreover, a method relying on maximal regularity theory  described in \cite{EscherSimonett1998_VPMCF} is used to prove that the convergence is exponential.  Indeed, its strength  allows us to extend the long time existence and convergence in Theorem \ref{main} to certain non-necessarily convex initial curves. With more precision,  we prove
 
\smallskip 
 \begin{thm}[Stability of {\sc ap-csf} near spheres]\label{corm} Let $\mc S$ be a geodesic sphere of a pinched Hadamard surface $\mathcal M$ and $0 < \beta < 1$. There exists an $\varepsilon>0$ such that, for every embedding  $\gamma_0:\mathbb S^1\rightarrow \mathcal M$ with $h^{1 + \beta}$-distance to $\mc S$ lower than $\varepsilon$, the  {\sc ap-csf} has a unique solution starting at $\gamma_0$, defined on $[0,\infty)$ and which  converges exponentially to a geodesic sphere in $\mathcal M$ $h^{1 + \beta}$-close to $\mc{S}$ and enclosing the same volume as $\Gamma_0$.
 \end{thm}
\noindent Here $h^{1+\beta}(\mathbb S^1)$ denotes the little H\"older space, i.e., the closure of $C^\infty(\mathbb S^1)$ in the  norm of $C^{1,\beta}(\mathbb S^1)$.

\subsection*{Main results in the perspective of the literature}

After the foundational planar work of Gage, Hamilton and Grayson \cite{Gage84, gage_hamilton_heat_equation, grayson1987heat}, Gage extended {\sc csf} to general surfaces \cite{gage_1990_csf_surfaces}, obtaining that initially closed and convex curves remain convex, and converge smoothly to a round point. Subsequent developments, notably by Grayson \cite{grayson1989shortening}, show that embedded curves in a surface with a Riemannian metric that is convex at infinity either contract to a point or subconverge to simple closed geodesics. Recent works study {\sc csf} even on surfaces with conical singularities \cite{rosa}.

In contrast, constrained non-local variants such as Gage’s {\sc ap-csf} are well developed mostly in the planar case, with global existence and convergence to circles for convex curves \cite{Gage1986}. There is also a growing literature on asymptotic circularity under total curvature \cite{Dit}, or star-shapedness \cite{star} assumptions, and {\sc ap-csf} for convex plane curves in an inhomogeneous medium \cite{Ni}. On general Riemannian surfaces, by comparison, there is essentially no previous literature, apart from some  advances in the hyperbolic plane \cite{wei_yang_VPF_Hyperbolic}. This leaves a substantial gap (which is challenging enough to motivate the present work) between the now classical unconstrained {\sc csf} on Riemannian surfaces and the constrained flows that are particularly natural for isoperimetric-type applications.

\smallskip  

In turn, natural higher-dimensional analogues of  \eqref{eq:ccf_f} are the volume-preserving mean curvature flow ({\sc vp-mcf}) and the area-preserving {\sc mcf} ({\sc ap-mcf}), which preserve the hypersurface enclosed volume and the area, respectively. The {\sc vp-mcf} was introduced by Huisken \cite{Huisken_1987_VPMCF} for strictly convex hypersurfaces of $\mathbb R^{n+1}$, showing that 
the flow has a solution on $[0,\infty)$, which stays convex for all time and converges exponentially to a round sphere. The analogous result for the area-preserving case can be found in \cite{McCoy-AP}.

\smallskip
In \cite{Huisken_1987_VPMCF}, Huisken highlighted the challenges involved in preserving convexity during the {\sc vp-mcf} when the ambient space is a general Riemannian manifold. In particular, he constructed examples of convex hypersurfaces in the sphere $\mathbb{S}^{n+1}$ that lose their convexity along the flow, stressing that convexity preservation cannot be taken for granted outside the Euclidean setting. 

\smallskip

Inspired by this example, the second author and Miquel  studied in  \cite{CabezasRivasMiquel2007} the {\sc vp-mcf} in the hyperbolic space $\mathbb H^{n+1}$. Assuming a stronger notion of convexity defined via horospheres (known as $h$-convexity), the flow exists for all time and the evolving hypersurface converges smoothly and exponentially to a geodesic sphere in $\mathbb H^{n+1}$, enclosing the same volume as the initial hypersurface. The $h$-convexity condition (which amounts to $\kappa_i > 1$ for all principal curvatures) was relaxed to positive sectional curvature ($\kappa_i \kappa_j > 1$ for $1\leq i \neq j \leq n$) in \cite{AnW-J}. Moreover, \cite{AndrewsWei2018} developed quermassintegral-preserving curvature flows in $\mathbb{H}^{n+1}$ and proved long-time existence and convergence results for $h$-convex hypersurfaces. Notice that for curves in $\mathbb H^2$ the $h$-convexity condition would mean $\kappa >1$, and it was relaxed to convexity ($\kappa >0$) successfully in \cite{wei_yang_VPF_Hyperbolic}.

\smallskip

  However, little is known in the setting of general negatively curved surfaces where the Gauss curvature is variable but controlled. This motivates our interest of extending \cite{wei_yang_VPF_Hyperbolic} to non-constant negatively-curved ambient surfaces.

\subsection*{Organization of the paper} The paper is organized so as to move from the most general setting to progressively more specific ambient geometries. Section \ref{basic} gathers the main evolution equations and monotonicity properties for length, area and isoperimetric deficit. The typical result that the flow exists forever unless the curvature becomes unbounded is carried out in Section \ref{long-time}. This part has the extra complication that Gauss curvature and its derivatives appear in all the evolution equations and one has to keep track of which orders of derivatives appear in each step in order to show that if the curvature is bounded, then all the higher order derivatives are also under control (cf.~Proposition \ref{Prop:bounds_higher_deriv_kappa_APCSF_LPCF}). Section \ref{Hadamard} contains geometric properties for Hadamard surfaces that are flow independent, such as a new Hessian comparison result or  isoperimetric estimates.  The preservation of convexity and the proof of Theorem \ref{main} (1) is done in Section \ref{convex}, while Section \ref{radius} establishes some uniform control for the inner and outer radius of the evolving domain under the flow. In turn, Section \ref{curvature} is devoted to the support function and the curvature estimates that lead to the proof of long-time existence. Finally, Section \ref{rotation} focuses on the area-preserving flow in the rotationally symmetric case, where we prove the convergence dichotomy and the stability result near geodesic circles.

\section{Basic properties of the flow } \label{basic}
Let $(\cM,g)$ be any Riemannian surface, $\gamma: \mathbb{S}^1 \times [0,\mathcal{T}) \to \cM $  evolving under
\begin{equation}
 \partial_t \gamma  =  F N, \quad \text{with} \quad F:=  h(t) - f(\kappa),
\label{eq:ccf_f}
\end{equation}
and $h(t)$ is chosen to keep area or length constant; more precisely,
\begin{equation} 
	h(t)=\frac{\int_{\mathbb{S}^1} f(\kappa)\kappa^{\alpha} \ ds}{\int_{\mathbb{S}^1} \kappa^\alpha \ ds}.
	\label{eq:def_h_LPCF_ff}
\end{equation}
 Here $f$ is a non-decreasing function of the curvature $\kappa$,
  $\Gamma_t=\gamma(\mathbb{S}^1, t)$, and $v= |\partial_p \gamma|$. We denote $\partial_s = \frac{\partial }{\partial s}$, $\partial_p =\frac{\partial }{\partial p}$, and $\partial_s^{n} = \frac{\partial^{n}}{\partial s^{n}}$.
  
   Recall that the spatial derivative of the unit tangent vector $T = \partial_s \gamma = \frac1{v}\partial_p \gamma$ gives 
\begin{equation} \label{Frenet}
    \partial_s T = - \kappa N.
\end{equation}
   Notice that, as vector fields over $\cM$, $\partial_t$ and $\ps$ should be regarded as $\partial_t \gamma$ and $T$, respectively. By routine computations as in \cite{Gage1986}, we get
\begin{lemma} \label{evol_eq}
    Under the flow \eqref{eq:ccf_f} we have the following evolution equations:
    \begin{enumerate}
        \item[{\rm (a)}] $\partial_t v = F \kappa v$.
        \item[{\rm (b)}] $\nabla_t T =  \partial_s F \, T$ and $\nabla_t N = - \partial_s F \, T.$
    \item[{\rm (c)}] The evolution of the length element $ds = v \, d p$ is given by $\partial_t ds = F \kappa \, ds$.              
    \item[{\rm (d)}] The commutator of the partial derivatives is  $\left[ \partial_s, \partial_t\right] = \kappa F \partial_s$.
    \item[{\rm (e)}] The curvature evolves under $\partial_t \kappa \, = \, - F \left( \KGauss + \kappa^2 \right) - \partial^2_s F$,
    \item[{\rm (f)}] The evolution of the curve’s length $L(t)$ and the enclosed area are given by 
        \begin{equation}
           \frac{d }{dt} L = \int_{\mathbb{S}^1} \kappa F \; ds \quad \text{and} \quad \frac{dA}{dt}= \int_{\mathbb{S}^1} F \ ds,
            \label{eq:lengthevol}
        \end{equation}
    where $\Omega_t$ is the region enclosed by $\Gamma_t$ and  $A(t)$ means the area of $\Omega_t$. 
    \end{enumerate}
\end{lemma}

The PDE in \eqref{eq:ccf_f} is not strictly parabolic, so standard  theory does not apply directly. Nevertheless, the flow is invariant under tangential perturbations (as in \cite[Proposition 1.3.4]{mantegazza_LNMCF}), which allows reformulation as a strictly parabolic system and ensures short time existence and uniqueness of solutions, up to reparametrization. 

As a direct consequence of the definition of the constrained flows and by means of a generalized Hölder inequality (cf. \cite{wei_yang_VPF_Hyperbolic}), one can prove
\begin{prop}\label{prop:APCSF_L_decreases_LPCF_A_increases} 
 Along the flow \eqref{eq:ccf_f}, if $h(t)$ is chosen  so that the enclosed area is preserved, then the length of the curve decreases with time. Conversely, for $h(t)$ taken so that the length is preserved, then the enclosed area increases.
\end{prop}

\begin{cor}[Monotonicity of the isoperimetric deficit] \label{lemma:monotonicityisoperimetricdeficit}
	Under \eqref{eq:ccf_f} with $h(t)$ as in \eqref{eq:def_h_LPCF_ff} for $\alpha =0,1$, the isoperimetric deficit 
	\begin{equation}
	\Delta(t):=L(t)^2-4\pi A(t) - a^2 A(t)^2 ,
	\end{equation}
	is monotone non-increasing.
	\begin{proof}
		In the area-preserving case, as length is non-increasing by Proposition \ref{prop:APCSF_L_decreases_LPCF_A_increases}, it holds
		$$  \Delta'(t) = 2 L(t)  L'(t) \leq 0.$$
		In turn, when the length is preserved, the area is non-decreasing.  Therefore,
		$$
		\Delta'(t) = -4\pi  A'(t) - a^2 A(t)   A'(t) \leq 0.
		$$

	\end{proof}
\end{cor}

\section{Results on surfaces with some assumptions on the Gauss curvature} \label{long-time}

\subsection{Bounds on the higher derivatives of the curvature and the Gauss curvature}
 
     Our argument to get higher-order derivative estimates of $\kappa$ and $\KGauss$ under  the flow \eqref{eq:ccf_f} follows the usual approach, but we have to  control extra terms due to the ambient Gauss curvature, which is neither zero nor constant as in the previous literature. Hence  the evolution equations for the derivatives of $\kappa$ involve derivatives of $\KGauss$ as well, and it becomes essential to track carefully which orders of derivatives of $\kappa$ and $\KGauss$ appear in each term.

 Hereafter $\Box:= \partial_t -\partial_s^2$ denotes the heat operator.  Arguing by induction and exploiting the formulas in Lemma \ref{evol_eq} (d) and (e), we reach    
    \begin{lemma}[Evolution equations for spatial derivatives of $\kappa$] \label{prop:evolutionequationskappa} 
        Let $\gamma:\mathbb{S}^1\times [0,\mathcal{T})\to \cM$ be a solution of \eqref{eq:ccf_kappa} with $h(t)$ as in \eqref{eq:def_h_LPCF_f} for any $\alpha$. Then the $\ell$-th spatial derivative of the curvature satisfies
       
           \begin{equation}
           \Box \, \partial_s^{\ell}\kappa = \partial_s^{\ell}\kappa + \left[\left( (\ell+3)\kappa - (\ell+2)h\right)\kappa \  + \KGauss \right] \partial_s^{\ell}\kappa + P_{\ell-1}  . 
           \end{equation}
        \label{eq:evolhigherderivcurvatura}
        where $P_\ell=P_\ell(h, \kappa, \partial_s \kappa, ... , \partial_s^\ell \kappa, \partial_s \KGauss, ..., \partial_s^{\ell+1} \KGauss)$ and $P_0= (\kappa - h)\ps \KGauss$.
 \end{lemma}

\begin{prop}[Higher derivatives of the Gauss curvature $\KGauss$] \label{prop:derivofKGauss} 
    Let $(\cM, g)$ be a  surface and $\gamma:\mathbb{S}^1 \to \cM$. For $n\geq 1$, set $\Theta:= \prod_{j=1}^{N_{mv}^{n}} \partial_s^{i_j} \kappa$, then
    \begin{equation}
        \nabla^n \KGauss (\underbrace{\partial_s, ..., \partial_s}_{n\text{-times}}) = \partial_s^n\KGauss + \sum_{I+m < n}\sum_{v\in V_m} C_{mv}^n \Theta \nabla^m \KGauss (v) ,
        \label{eq:ncovderivKGauss}
    \end{equation}
    where $I= \sum_{j=1}^{N_{mv}^n} i_j$, with $i_j \in \mathbb N_0= \mathbb N \cup \{0\}$; $C_{mv}^n$ is a constant depending on $m$, $v$ and $n$, which can be zero; $N_{mv}^n \in \mathbb N_0$, with $N_{mv}^n < n$. If $N_{mv}^n = 0$, then $\Theta = 1$ and $I = 0$, and $V_m := \{ \partial_s , N \}^m$.   Thus, $\nabla^m \KGauss(v)$ is a shortcut to denote $\nabla^m\KGauss$ acting on $m$ copies of $\ps$ and/or $N$.

    \begin{proof}
     Arguing by induction, the case $n=1$ is deduced noticing that  $\KGauss\in C^\infty(\mathcal M)$, then
     \begin{equation}
         \nabla \KGauss (\partial_s) = \nabla_{s} \KGauss = \partial_s \KGauss ,
     \end{equation}
     where $\nabla_s = \nabla_T$. Now let us assume \eqref{eq:ncovderivKGauss} for $m\leq n$. We will prove it for $n+1$: by the definition of the covariant derivative of a tensor field and Frenet formula, we reach
    \begin{align}
        \nabla^{n+1} \KGauss (\partial_s, ... ,\partial_s) 
        &= \nabla_s \left(\nabla^n \KGauss (\partial_s,..., \partial_s) \right) - \sum_{i=1}^n \nabla^n \KGauss (\partial_s,..., \underbrace{-\kappa N}_{i\text{-th position}},..., \partial_s),
    \end{align}
where the last term already has the sought form. Then the induction hypothesis  yields
    \begin{align}
       \nabla_s \left(\nabla^n \KGauss (\partial_s,..., \partial_s)\right) & = \partial_s^{n+1}\KGauss + \sum_{I+m < n}\sum_{v\in V_m} C_{mv}^n \ \Big(\partial_s\Theta \nabla^m \KGauss (v) + \Theta \nabla_s \nabla^m \KGauss (v)\Big).
    \end{align}
    By the product rule, $\partial_s\Theta$ produces a finite sum of $N_{mv}^{n}$ terms of the same type, with the multi-index $\{i_1,...,i_{N_{mv}^{n}}\}$ updated by adding $1$ to one of the entries each time. Hence the expression for the term involving $\partial_s\Theta$ is of the desired form.
    
     Moreover, as before, we get
    \begin{align}
        \nabla_s \nabla^m \KGauss (v) & =  \nabla^{m+1} \KGauss (v) + \sum_{i=1}^{m}  \nabla^m\KGauss (v_1,..., \partial_s v_i, ..., v_{m}) = \nabla^{m+1} \KGauss (v) \pm \sum_{i=1}^{m}   \kappa \nabla^m \KGauss(\widetilde{v}) ,
    \end{align}
where $\widetilde v$ denotes the corresponding vector obtained from $v$ by replacing one entry, and the sign depends on the Frenet relation used. Thus this expression is again of the desired form, which joining all the terms concludes the induction.

Observe that, for each $n+1$, the constants $C_{mv}^{n+1}$, which depend on both $m$ and $v$, are uniformly bounded by a constant  depending only on $n+1$. Indeed, each $C_{mv}^{n+1}$ is obtained by combining finitely many constants from the previous induction step. Therefore, the summation remains finite at every stage, even if the combinations of terms become increasingly complex.
\end{proof}

\end{prop}

Throughout this subsection, assume that the ambient surface $\cM$ satisfies
\begin{equation} \label{K_bdd}
    \sup_{0\leq \ell \leq n} \big| {\nabla}^\ell \KGauss \big| \leq {C}_n
\end{equation}
for some $n \in \mathbb{N}$, where ${C}_n$ are fixed constants. Our next  aim  is to obtain bounds for $|\partial_s^n \KGauss|$, and also to bound higher order spatial derivatives of $\kappa$ (exactly up to order $n$, see Proposition \ref{Prop:bounds_higher_deriv_kappa_APCSF_LPCF}).

\begin{cor}[Bounds of the higher spatial derivatives of  $\KGauss$]\label{prop:bounds_of_deriv_KGauss}
Let $\gamma:\mathbb{S}^1 \to \cM$ satisfying
\begin{equation}
     \sup_{0\leq \ell \leq n-1}\big| \partial_s^\ell \kappa \big| \leq \overline{\kappa}_{n -1}
\end{equation}
for constants $\overline{\kappa}_{n-1}$ depending on $n$ and the initial curve. Then for each $n\geq 1$ it holds
\begin{equation}
    \left| \partial_s^n \KGauss \right| \leq   \widetilde{C}_n,
\end{equation}
where $\widetilde{C}_n$ denotes a constant depending on $n,\overline{\kappa}_0$ and $\gamma_0$.
\begin{proof}
Using the formula \eqref{eq:ncovderivKGauss} and that the constants $C_{mv}^n$ are uniformly bounded, we reach
\begin{align}
            | \partial_s^n\KGauss  | & \leq \left| \nabla^n \KGauss (\partial_s, ..., \partial_s) \right| + \sum_{I+m < n}\sum_{v\in V_m} |C_{mv}^n| |\Theta| \left|\nabla^m \KGauss (v) \right| \leq {C}_n + c(n) \left(   \overline{\kappa}_{n -1} \right)^I  {C}_n
            \\ & \leq C_n\big(1+c(n) \max\{1, \left(\overline{\kappa}_{n -1} \right)^n\}\big) =:\widetilde{C}_n.
\end{align}
\end{proof}
\end{cor}

We aim to obtain $t$-independent upper bounds for $|\partial_s^n \kappa|$, $n \geq 1$, any constrained curvature flow (that is, taking any value of $\alpha$ in the definition of $h(t)$).

\begin{prop}[Bounds of the higher derivatives of $\kappa$]\label{Prop:bounds_higher_deriv_kappa_APCSF_LPCF}
Let $\gamma:\mathbb{S}^1 \times [0,\mathcal{T}) \to \cM$ be a solution of the flow \eqref{eq:ccf_kappa}. Suppose  that there is a uniform constant $\overline{\kappa}_0 > 0$ such that
\begin{equation}
    |\kappa| (p,t) \leq \overline{\kappa}_0, \qquad \text{for all } (p,t)\in \mathbb{S}^1\times [0,\mathcal{T}).
\end{equation}
Then, there exists a constant $\overline{\kappa}_n = \overline{\kappa}_n(n,\gamma_0, \overline{\kappa}_0, \widetilde{C}_n)$ such that
\begin{equation}
    \sup_{1\leq \ell \leq n} \big| \partial_s^\ell \kappa \big|(p,t) \leq \overline{\kappa}_n, \qquad \text{for all } (p,t)\in \mathbb{S}^1\times [0,\mathcal{T}).
\end{equation}
\end{prop}

 \begin{proof}
  For simplicity, we will denote by $\mathbf{C}$ any  constant depending on  $n$, $\gamma_0$, and $\overline{\kappa}_{0}$; despite the explicit meaning may change from line to line.  We prove the claim by induction:  by assumption,  the hypotheses of Corollary \ref{prop:bounds_of_deriv_KGauss} are satisfied for $n=1$, and we obtain
\begin{equation}
    |\ps \KGauss | \leq \widetilde{C}_1(\overline{\kappa}_0, \gamma_0).
\end{equation}

   Let $\Lambda = \Lambda(\overline{\kappa}_0, C_0) >0$ be a  constant to be determined later, and set $\mathscr F := (\ps \kappa)^2 + \Lambda \kappa^2$. By definition of $h$, we get $|h(t)|\leq \overline{\kappa}_0$. The latter, together with Lemma \ref{evol_eq} and Lemma \ref{prop:evolutionequationskappa},  yields
 \begin{align*}
     \Box \mathscr F = & -2 \left(\ps^2\kappa\right)^2  + 2 \ps\kappa\left[\left( 4\kappa - 3h\right)\kappa \  \ps\kappa + \KGauss \ \ps\kappa + P_0 \right] +2\Lambda\kappa (\kappa -h)(\KGauss+\kappa^2) - 2 \Lambda\left(\ps\kappa\right)^2 \\
      \leq & \;\ 2  \left( 7 \, \overline{\kappa}_0^2  + {C}_0 - \Lambda \right) \left( \ps\kappa\right)^2   + 4 \, \overline{\kappa}_0 \widetilde C_1 |\ps \kappa|    +4\Lambda\, \overline{\kappa}_0^2 (\widetilde C_0 +\overline{\kappa}_0^2) \\
      \leq & \;\ 2  \left( 7 \, \overline{\kappa}_0^2  + {C}_0 - \Lambda + 1\right) \left( \ps\kappa\right)^2   + 4 \, (\overline{\kappa}_0 \widetilde C_1)^2 +  \mathbf{C},
 \end{align*}
where we have applied Young's inequality. Then, setting $\Lambda( \overline{\kappa}_0, \widetilde{C}_n) : = 7 \, \overline{\kappa}_0^2  + {C}_0 + \frac{3}{2}$, we reach
  \begin{align}
     \Box \mathscr F & 
     \leq  - \left( \partial_s\kappa\right)^2   + \mathbf{C}
     =  - \left( \partial_s\kappa\right)^2 - \Lambda \kappa^2 + \Lambda \kappa^2   + \mathbf{C}      \leq  - \mathscr F  + \mathbf{C}_1.
     \label{eq:operadordelcalorfitakamblambdafixat}
 \end{align}
Assume, that $\mathscr F$ reaches a maximum value $K > \max \{ \max_{p \in \mathbb{S}^1} \mathscr F(\cdot,0) , \ \mathbf{C}_1 \}$
for the first time. Then, at that time $0 \leq \Box \mathscr F \leq - K + \mathbf{C}_1 < 0$, which is a contradiction. Hence, for all $t\in [0,\mathcal{T})$
\begin{equation}
        \max_{  \mathbb{S}^1}|\partial_s \kappa  | (\cdot,t)
        \leq \max_{ \mathbb{S}^1}\mathscr F(\cdot,t)  
        \leq \max \{ \max_{  \mathbb{S}^1}\mathscr F(\cdot,0) , \ \mathbf{C}_1\} 
         = \mathbf{C}_{1}^*( \gamma_0,\overline{\kappa}_0, \widetilde{C}_1 ),
\end{equation}
which implies
\begin{equation}
    \sup_{0\leq \ell \leq 1} | \partial_s^\ell \kappa| (p,t) \leq \max \{ \mathbf{C}_{1}^* , \overline{\kappa}_0 \} =: \overline{\kappa}_{1}.
\end{equation}

Now, by induction assume that $\sup_{0\leq \ell \leq m}\big| \partial_s^\ell \kappa   \big|\leq \overline{\kappa}_{m}$
for some $m$ with $1 \leq m < n$. To prove the claim for $m+1$, notice that since $m < n$, by hypothesis we have
 \begin{equation}
 \sup_{0\leq \ell \leq m+1} \big| {\nabla}^\ell \KGauss \big| \leq {C}_n    .
  \label{eq:induction_hyp_2_fitar_deriv_kappa}
 \end{equation}
Therefore Corollary \ref{prop:bounds_of_deriv_KGauss} leads to
\begin{equation}
    |\ps^{\ell}\KGauss| \leq   \widetilde{C}_{\ell} \qquad \text{for all} \quad 1\leq \ell \leq m+1
     \label{eq:induction_hyp_3_fitar_deriv_kappa}
\end{equation}

Now let $\Lambda$ be again a positive constant to be chosen later, and set $\mathscr G :=(\partial_s^{m+1}\kappa)^2 + \Lambda (\partial_s^m \kappa )^2$. Using the previous bounds and Lemma~\ref{prop:evolutionequationskappa}, we have
 \begin{align}
     \Box \mathscr G       \leq & \  2 \left[\left( (m+4)\kappa - (m+3)h\right)\kappa \ + \widetilde{C}_n  - \Lambda \right] (\partial_s^{m+1}\kappa)^2   + 2\,\partial_s^{m+1}\kappa \ 
      P_m  +  \Lambda \big[ \mathbf{C} + P_{m-1} \big],
 \end{align}
where each $P_k$ depends polynomially on the terms
\begin{equation}
    h, \kappa, \partial_s \kappa, \ldots, \partial_s^k \kappa \quad \text{and} \quad \partial_s \KGauss, \ldots, \partial_s^{k+1} \KGauss.
\end{equation}
In particular, they are uniformly bounded by the induction hypothesis and \eqref{eq:induction_hyp_3_fitar_deriv_kappa}. The rest of the argument to achieve the desired bounds is completely analogous to the case $m = 1$. 
 \end{proof}

\subsection{Long time existence}
We show that under \eqref{eq:ccf_kappa}  the solution can only develop singularities if the curvature becomes unbounded. The following theorem formalizes this criterion and its proof adapts the approach of~\cite[pp.~257ff]{Huisken84}, with modifications arising from the non-constant Gauss curvature. Accordingly, we will explicitly indicate just the parts that differ from Huisken’s original proof.

\begin{thm}\label{teo:k_unbdd_or_T_infty}
	Let $(\cM,g)$ be a Riemannian surface satisfying
	\begin{equation}
	\sup_{\mathcal M} |\KGauss | \leq {C}_0
	\end{equation}    
	and let  $\gamma:\mathbb{S}^1 \times [0,\mathcal{T}) \to \cM$ be a solution of the flow \eqref{eq:ccf_kappa} with $h(t)$ as in  \eqref{eq:def_h_LPCF_f}  for any $\alpha$ and $[0,\mathcal{T})$ the maximal time interval where the solution exists. Then either
	\begin{align}
	\mathcal{T}=\infty \qquad \text{or} \qquad \max_{\Gamma_t}|\kappa| \quad \text{ becomes unbounded as }t\to \mathcal{T}.
	\end{align}
\end{thm}
 
\begin{proof}
    Suppose $\mathcal{T} < \infty$ and that  we can find some constant $\overline{\kappa}_0< \infty$ such that
    \begin{equation}
        \max_{\mathbb{S}^1}|\kappa|(\cdot,t) \leq \overline{\kappa}_0 \quad \text{on} \quad [0,\mathcal{T}).
        \label{eq:condiciokfitada}
    \end{equation}
    Then for any constrained flow $|h(t)|\leq \overline{\kappa}_0$ and for all $p\in \mathbb{S}^1$ and $\tau >\frak t$ in $(0,\mathcal{T})$, we have
     \begin{align}
        \text{dist}_{g}\big(\gamma(p,\tau), \gamma(p,\frak t)\big) 
        & \leq \int_\frak t^{\tau} \left| \partial_t \gamma\right| dt 
        = \int_\frak t^{\tau} \left|\left( h  -\kappa\right) N(p,t) \right| dt \leq 2 \overline{\kappa}_0 (\tau-\frak t).
        \label{eq:argumentconvergenciacorbaC0}
    \end{align}
   From here $\gamma(\cdot,t)$ converges uniformly to some continuous limit function $\gamma(\cdot,\mathcal{T})$ as $t\to \mathcal{T}$, and
         $\Gamma_t$ stays in a compact region of $\cM$;  thus $\max_{0\leq \ell \leq m} \left| \nabla^m \KGauss \right| \leq {C}_m$  for all $m$.

Now,  the evolution equation from Lemma \ref{evol_eq} (a) and the upper bound for the curvature yield
\begin{equation}
    \partial_t v = (h - \kappa)\kappa \, v \geq  -2 \overline{\kappa}_0^2 v.
\end{equation}
By Grönwall’s Lemma, for all $t \in [0, \mathcal{T})$ we obtain $v(t) \geq v(0) \, e^{-2\overline{\kappa}_0^2 t} > 0$,
where the latter uses that $\gamma_0$ is an immersion. Hence, $v(\mathcal{T}) > 0$, and    $\gamma(\cdot,\mathcal{T})$ actually represents an immersed curve $\Gamma_{\mathcal{T}}=\gamma(\mathbb{S}^1,\mathcal{T})$. It remains to show that such a limiting curve is $C^\infty$.
 
   To achieve this, we will use an argument similar to the one in \eqref{eq:argumentconvergenciacorbaC0} but with higher order derivatives. Observe that we can apply Proposition \ref{Prop:bounds_higher_deriv_kappa_APCSF_LPCF} and Corollary  \ref{prop:bounds_of_deriv_KGauss} to deduce
    \begin{equation}
         \sup_{0\leq \ell \leq m} \big| \partial_s^\ell \kappa   \big|(p,t) \leq \overline \kappa_m \quad \text{and} \quad   \sup_{0\leq \ell \leq m+1}|\partial_s^{\ell}\KGauss|(p,t)  \leq \hat{C}_{m}.
    \end{equation}
    for all $(p,t)\in \mathbb{S}^1\times [0,\mathcal{T})$ and $m\geq 1$.
    
Now routine computations lead to complicated formulas relating higher order derivatives of $\gamma$ with those of $\kappa$ and $\KGauss$, which can be summarized as
     \begin{align}
        \nabla_t \nabla_s^{m} \gamma  = -\partial_s^{m}\kappa \ N  & + \widetilde{P}_1 (h, \kappa, \KGauss, \partial_s \kappa, \dots, \partial_s^{m-2}\kappa, \partial_s \KGauss, \dots, \partial_s^{m-2}\KGauss) N\\
        &      + \widetilde{P}_2 (h, \kappa, \KGauss, \partial_s \kappa, \dots, \partial_s^{m-1}\kappa, \partial_s \KGauss, \dots, \partial_s^{m-3}\KGauss) T,
    \end{align}
   where $\widetilde{P}_{i}$ are polynomials that depend on the terms in the brackets, with coefficients uniformly bounded. These expressions allow us to obtain an upper bound for $\nabla_t \partial_s^m \gamma$, independent of $t$,  for all $m\geq 1$. From here the rest of the argument follows exactly as in \cite{Huisken84}.
\end{proof}

\subsection{Uniform convergence of the curvatures} 
In this section, we bring together the tools developed so far to prove Theorem \ref{main} (3). More generally,
\begin{thm}\label{teo:limit_kappa_es_h} 
	Let $(\cM, g)$ be a surface and  $\gamma: \mathbb{S}^1 \times [0, \mathcal{T}) \to \cM$ be a smooth family of embedded closed curves evolving under the flow \eqref{eq:ccf_kappa} with global term $h(t)$ defined in \eqref{eq:def_h_LPCF_f} for $\alpha =0,1$. Assume that
	\begin{equation}
	\max_{\Gamma_t}\kappa \leq \overline{\kappa}_0 \quad \text{and} \quad  \sup_{\Gamma_t}|\nabla^{\ell}\KGauss|<{C}_2 \quad \text{for} \quad 0\leq \ell \leq 2.
	\label{eq:condicions_tma_limit_kappa_es_h}
	\end{equation}
	In addition, in the case of the length-preserving flow, suppose further that
	\begin{equation}
	0< a^2 \leq \KGauss \qquad \text{or } \qquad \KGauss \leq  -a^2 <0 \quad \text{with } \mathcal M \text{ simply connected}.
	\label{eq:condicions_tma_limit_kappa_es_h_extra_LP}
	\end{equation} 
	Then, $\mathcal{T}=\infty$ and
	\begin{equation} 
	\lim_{t\to \infty} \sup_{\mathbb{S}^1} \left| \kappa(\cdot,t) - h(t) \right| = 0.
	\label{eq:limit_kappa_es_h}
	\end{equation}
\end{thm}

\begin{bem}
Notice that, for the {\sc ap-csf} in the flat case, \eqref{eq:limit_kappa_es_h}  implies that the curvature $\kappa$ of the evolving curves $\gamma_t$ converges uniformly to the constant $h_\infty:=\lim_{t\to \infty}h(t)$. Indeed, in this case $h(t)$ is a bounded monotone function, whereas for the non-flat curvature case $h'$ does not  necessarily have a constant sign, and we can only guarantee subconvergence because it is easy to show that $h(t)$ is uniformly bounded. The condition on the derivatives of the Gauss curvature would be automatically satisfied if the evolving curves stay within a compact region of $\mathcal M$, as happens in Theorem \ref{main} (4), item (a).
\end{bem}

\begin{proof}
By hypothesis, we have $    \max_{\Gamma_t} \kappa \leq \overline{\kappa}_0$. Therefore, we can apply Theorem \ref{teo:k_unbdd_or_T_infty}, which ensures that the solution exists for all time, i.e., $ \mathcal{T} = \infty $. Now consider the $ L^2 $-energy given by
\begin{equation}
E(t) := \int_{\mathbb{S}^1} (\kappa - h)^2 \, ds.
\end{equation}
We aim to prove that $\int_0^\infty E(t) \, dt$
is finite for both flows. Indeed, for the {\sc ap-csf},  \eqref{eq:lengthevol} leads to 
\begin{equation}
     L'(t) = \int_{\mathbb{S}^1} \kappa (h-\kappa) \, ds = - \int_{\mathbb{S}^1} (\kappa - h)^2 \, ds = - E(t).
\end{equation}
Integrating in time yields
\begin{equation}
    \int_0^\infty E(t) \, dt  = L_0 - \lim_{t\to \infty} L(t) \leq L_0.
    \label{eq:bound_integral_energy_AP}
\end{equation}

On the other hand, the argument for the {\sc lp-cf} requires a different approach, since the curve length is constant. Using the evolution equations  \eqref{eq:lengthevol}, we obtain
\[E(t)=  \int_{\mathbb{S}^1} \kappa(\kappa-h)  \, ds  -h \int_{\mathbb{S}^1}  (\kappa-h)  \, ds  \\
         =  - L' + h  A' =  h  A'.\]
     Using that  $A$ is non-decreasing (cf. Proposition~\ref{prop:APCSF_L_decreases_LPCF_A_increases}), we find  $\int_0^\infty E(t) \, dt \leq \overline \kappa_0 \lim_{t \to \infty} A(t)$. To find a uniform area bound, we consider the two cases in assumption~\eqref{eq:condicions_tma_limit_kappa_es_h_extra_LP}. If $0< a^2 \leq \KGauss$, then by Gauss–Bonnet theorem, we get
        \begin{equation}
            a^2 A(t) \leq \int_{\Omega_t} \KGauss \ dA = 2\pi - \int_{\mathbb{S}^1} \kappa \ ds \leq 2\pi + \overline \kappa_0 L_0\, .
        \end{equation}
In turn, if $\KGauss \leq -a^2 <0 $, an isoperimetric type inequality by Osserman \cite[Theorem 8]{ossermanIospIneq} gives
        \begin{equation}
            A(t) \leq \frac{L(t)}{a^2} = \frac{L_0}{a^2} \, .
        \end{equation} 
Hence $A(t)$ is bounded in either case.

Next, we compute the time derivative of $E(t)$ under \eqref{eq:ccf_kappa}. Using the evolution equations from Lemma \ref{evol_eq} (c) and (e), we reach
\begin{align}
\frac{d}{dt} E(t) 
& \leq  2 \int_{\mathbb{S}^1} (\kappa - h)^2 (\KGauss + \kappa^2) \, ds  - \int_{\mathbb{S}^1} (\kappa - h)^3 \kappa \, ds,
\end{align} 
where the inequality comes from $2 \int_{\mathbb{S}^1} (\kappa - h) \partial_s^2 \kappa  \leq -2 \int_{\mathbb{S}^1} (\partial_s \kappa)^2$. Notice that each term is uniformly bounded by hypothesis. Therefore,  it follows that
\begin{equation}
    \lim_{t \to \infty} E(t) = \lim_{t \to \infty} \int_{\mathbb{S}^1} (\kappa - h)^2 \, ds = 0.
    \label{eq:lim_Energia_0}
\end{equation}

Since $(\kappa - h)$ has zero mean on the compact manifold $\mathbb{S}^1$ and is smooth, it satisfies the hypotheses of \cite[Theorem 3.67]{Aubin1998}, so we have
\begin{equation}
    \sup_{\mathbb{S}^1} |\kappa - h| \leq C \| \nabla (\kappa - h) \|_2.
\end{equation}
Recalling that $\partial_s \kappa$ is uniformly bounded by Proposition  \ref{Prop:bounds_higher_deriv_kappa_APCSF_LPCF}, we can apply the interpolation inequality from \cite[Theorem 3.69]{Aubin1998}, for  $p = q = r = 2$, to attain
\begin{equation}
    \left( \sup_{\mathbb{S}^1} |\kappa - h| \right)^2 \leq C \| \partial_s (\kappa - h) \|_2^2 \leq C \| \kappa - h \|_2 \| \partial_s^2 \kappa \|_2 = C E(t)^{1/2} \| \partial_s^2 \kappa \|_2.
\end{equation}
Finally, as $\| \partial_s^2 \kappa \|_2$ is uniformly bounded by Proposition \ref{Prop:bounds_higher_deriv_kappa_APCSF_LPCF},   \eqref{eq:lim_Energia_0} ensures \eqref{eq:limit_kappa_es_h}.
\end{proof}

\section{Geometric properties of pinched Hadamard surfaces} \label{Hadamard}
Throughout this part we assume that $(\cM,g)$ is a pinched Hadamard surface  as in \eqref{pinched}, and fix $p \in \cM$. Consider the radial distance function $r = r_p :=\text{dist}_{g}(p,\, \cdot \,).$  We shall use the notation $\pr :=  \nabla r$ for the gradient of $r$, while $\pr^{\top}$  is the corresponding tangential component.

\subsection{Hessian comparison enhanced}

In this setting the Hessian Comparison (see e.g.~\cite[Theorem 6.4.3]{PetersenRiemGeom}) says that, for any tangent vector $X$, it holds
\begin{equation}
    a \coth(a r )  \left( | X |^2 - \avg{ X, \pr }^2 \right)
\leq
\operatorname{Hess}(r)(X, X) 
\leq
b \coth(b r )  \left( | X |^2 - \avg{ X, \pr }^2 \right).
\label{eq:hessian_r_comparison_general}
\end{equation}

Now, for each $R > 0$, the \emph{geodesic circle} of radius $R$ centered at $p$, that is,
\begin{equation}
S_R := \{ x \in M \mid r(x) = R \}
\label{eq:geodesic_circle_centered_p}
\end{equation}
 is a regular embedded curve in $\mathcal M$ with unit tangent vector field $U$ satisfying $\langle U, \pr \rangle = 0$. Hence its 
curvature, denoted by $\kcirc$, is  given by
\begin{equation}
\kcirc  =  \avg{\nabla_{U} \pr, U} = \operatorname{Hess}(r)(U, U).
\label{eq:def_kcirc}
\end{equation}
From here, \eqref{eq:hessian_r_comparison_general} applied to $U$ leads to
\begin{equation}
    a \coth(a r) \leq \kcirc \leq b \coth(b r).
    \label{eq:fita_curvatura_cercle_geodesic}
\end{equation}

Now we derive a more specific comparison result:
\begin{lemma}
	\label{lema:hessian_inequalities_r}
 Let $\gamma : \mathbb{S}^1 \to \cM$ be a smooth, closed and embedded  curve, such that $p\notin \gamma(\mathbb{S}^1)$. Then, at every point of the curve, the following inequalities hold:
	
	\begin{itemize}
		\item[(a)] Tangent inequality: $0 < a \coth(a r) \big(1 - |\pr^{\top}|^2 \big)
			\leq \langle \nabla_T \pr, T \rangle
			\leq b \coth(b r) \big(1 - |\pr^{\top}|^2 \big).$ \smallskip
		
		\item[(b)] Normal inequality: $0 < a \coth(a r) \, |\pr^{\top}|^2
			\leq \langle \nabla_N \pr, N \rangle
			\leq b \coth(b r) \, |\pr^{\top}|^2.$ \smallskip
		
		\item[(c)] $\langle \nabla_N \pr, T \rangle= \langle \nabla_T \pr, N \rangle$, and one gets a mixed term inequality:
	\end{itemize}
	\begin{align}
	\mathscr F_{a,b} (r) \leq \langle \nabla_N \pr, T \rangle  \leq \mathscr F_{b,a} (r),
		\label{eq:desigualtat_hessiaR_tangentinormal}
	\end{align}
where \[2 \, \mathscr F_{a,b} (r):= 	 a \coth(a r) (1 - 2 \langle T, \pr \rangle \langle N, \pr \rangle) -  b \coth(b r).\]
\end{lemma}

    \begin{proof}
The tangent and normal inequalities  follow directly from \eqref{eq:hessian_r_comparison_general}. To establish (c), using the definition of  $\kcirc$ in \eqref{eq:def_kcirc} and that $\nabla_{\pr} \pr = 0$, we get
\begin{align*}
\langle \nabla_N \pr, T \rangle &= \langle \nabla_{N-\langle N, \pr \rangle \pr} \pr, T \rangle = \langle\kcirc (N-\langle N, \pr \rangle \pr), T \rangle = -\kcirc \langle N, \pr \rangle \langle T, \pr \rangle.
\end{align*}
By the same argument, we also obtain
\begin{align}
\langle \nabla_T \pr, N \rangle &= - \kcirc  \langle T, \pr \rangle \langle \pr, N \rangle = \langle \nabla_N \pr, T \rangle.
\label{eq:avg_T_N_pr}
\end{align}

Now, applying \eqref{eq:hessian_r_comparison_general} to the vector field $X = T + N$, which is tangent to $\cM$, leads to
\begin{equation}
a \coth(a r) \cdot \left(1 - 2 \langle T, \partial_r \rangle \langle N, \partial_r \rangle \right) 
\leq \langle \nabla_{T+N} \partial_r, T + N \rangle 
\leq b \coth(b r) \cdot \left(1 - 2 \langle T, \partial_r \rangle \langle N, \partial_r \rangle \right),
\label{eq:desigualtat_previa_hessiar_mixed}    
\end{equation}
where we have used that $\langle X, \partial_r \rangle^2  = 1 + 2 \langle T, \partial_r \rangle \langle N, \partial_r \rangle$.

On the other hand, thanks to \eqref{eq:avg_T_N_pr}, we reach the expansion
\begin{align*}
\langle \nabla_{T+N} \partial_r, T+N \rangle &= 
 \langle \nabla_T \partial_r, T \rangle 
+ \langle \nabla_N \partial_r, N \rangle 
+ 2 \langle \nabla_T \partial_r, N \rangle.
\end{align*}
Plugging this into \eqref{eq:desigualtat_previa_hessiar_mixed}, the desired bounds follow by adding  the estimates in (a) and (b). 
\end{proof}

\subsection{Results that only require a negative lower bound}

 Let $\Omega$ be a bounded domain with boundary $\Gamma=\partial \Omega$ within a Hadamard surface $\mathcal M$.
Recall that the inner radius $\rho_-$ and outer radius $\rho_+$ of  $\Omega$ are defined by
\begin{equation}
    \rho_-  =  \sup \{ \rho \, \mid \, B(p,\rho) \subset \Omega \ \text{for some } 
 \ p\in \cM \} ,
    \label{eq:innerradius}
\end{equation}
and
\begin{equation}
    \rho_+  =  \inf\{\rho \, \mid \, \Omega \subset  B(p,\rho) \ \text{for some } \ p \in \cM \} ,
    \label{eq:outerradius} 
\end{equation}
where $B(p,\rho)$ denotes the metric geodesic ball centered at $p$ with radius $\rho$.

\begin{lemma} \label{Had_low}
	Let $(\cM, g)$ be a Hadamard surface with Gauss curvature $\KGauss \leq -a^2$, where $a > 0$.
	\begin{enumerate}
	\item[{\rm (a)}] $L \geq a A$.
    \item[{\rm (b)}] For the isoperimetric deficit it holds $\sqrt{\Delta} \geq \Big|L-A a \coth\left(\frac{a \rho_{-}}{2}\right)\Big|.$
    \item[{\rm (c)}] Any  embedded, convex, and closed curve $\gamma\colon \mathbb{S}^1 \to \cM$ has curvature strictly positive on a non-empty open subset of $\mathbb{S}^1$.
  \end{enumerate}
    \begin{proof}
    	(a) comes from \cite[Theorem 8 (a)]{ossermanIospIneq}. To get (b), we apply another inequality due to Osserman (see \cite[Theorem 6]{ossermanIospIneq}):
        \begin{equation}
            \Delta \geq \bigg( L - \frac{L_{\rho_-}^{-a^2}}{A_{\rho_-}^{-a^2}} A \bigg)^2,
        \end{equation}
        where $A_{\rho_-}^{-a^2}$ and  $L_{\rho_-}^{-a^2}$ denote the area and the length, respectively, of a geodesic disk of radius $\rho_-$ in the hyperbolic plane of constant curvature $-a^2$ . More explicitly, we have
$$L_{\rho_-}^{-a^2}= 2\pi \frac{\sinh(a \rho_{-})}{a} 
\qquad
\text{and}
\qquad
A_{\rho_-}^{-a^2}= \frac{4\pi}{a}  \sinh^2\left(a\frac{\rho_{-}}{2}\right) = 2\pi \frac{\cosh(a\rho_{-})-1}{a^2}.$$
From here (b) follows by direct substitution. Now (c) is a  by-product of Gauss–Bonnet, since
\begin{equation}
\int_{\mathbb{S}^1} \kappa \, ds = 2\pi + \int_{\Omega} (-\KGauss) \, dA \geq 2\pi + a^2 A(\Omega) > 0,
\end{equation}
As by hypothesis $\kappa \geq 0$, the integral being strictly positive implies that $\kappa > 0$ on a subset of positive measure. In particular, there exists a non-empty open subset of $\mathbb{S}^1$ where $\kappa >0$.
\end{proof}
\end{lemma}

\section{Preservation of convexity} \label{convex}

In this section, we study the preservation of convexity for curves solving the constrained flow \eqref{eq:ccf_f}, under the assumption that $f$ is non-decreasing, which is crucial for short time existence.

\begin{lemma}[Preservation of convexity]\label{lema:preservationofstrictconvexity}
     Let $(\cM, g)$ be a pinched Hadamard surface.
    Suppose the initial curve $\gamma_0$ is smooth, closed, and (strictly) convex. Then, the solution $\gamma(\cdot, t)$ of the flow \eqref{eq:ccf_f}, with $f$ non-decreasing and any $h(t)$ as in \eqref{eq:def_h_LPCF_ff}, remains (strictly) convex for all $t > 0$ as long as the solution exists. More precisely,
       \begin{equation}
        \kappa(p,t)\geq \min \{ \min_{\mathbb{S}^1} \kappa(\cdot, 0), a \} := \varepsilon_0 \geq 0 \qquad \text{ for all $p \in \mathbb{S}^1$ and all $t \in [0, \mathcal{T})$.}
    \end{equation}

\begin{proof}
    As $\gamma_0$ is strictly convex, by continuity $\kappa$ remains positive on $\Gamma_t$ for short time. For each $t_0>0$, let $p_0\in \mathbb{S}^1$ be the point with
    \begin{equation}
        \kappa(p_0,t_0)= \min_{\mathbb{S}^1} \kappa(\cdot, t_0) =: \kappa_{\min}(t_0) .
    \end{equation}
    Then $\partial_s \kappa =0$ and $\partial_s^2 \kappa \geq 0$ at $(p_0,t_0)$. Since $f$ is non-decreasing and $\kappa$ attains its minimum at $(p_0, t_0)$, it follows that $f(\kappa)$ also does, therefore 
    \begin{equation}
        \partial ^2_s \left(f(\kappa)\right)\big|_{(p_0,t_0)}\bigg. \geq 0 \qquad \text{and} \qquad
        h(t) = \frac{\int_{\mathbb S^1} \kappa^\alpha f(\kappa) \, ds}{\int_{\mathbb S^1} \kappa^\alpha \, ds}\geq  f(\kappa_{\min}(t_0)).
    \end{equation}

 If $\kappa^2_{\min}(t_0)  \geq a^2>0$, we have nothing to prove, as the statement follows from $\varepsilon_0=a$. Otherwise,  if $\kappa^2_{\min}(t_0) < a^2 $, by the evolution equation  from Lemma \ref{evol_eq} (e), we get 
\begin{equation}
    \kappa_{\min}'(t_0) \geq (h-f(\kappa_{\min}(t_0)))(-\KGauss-\kappa^2) \geq (h-f\kappa_{\min}(t_0))(a^2-\kappa_{\min}^2(t_0)) >0.
\end{equation}
     From here, the maximum principle  leads to the claimed inequality.
\end{proof}
\end{lemma}

We prove that the curvature evolution under the area- or length-preserving flow \eqref{eq:ccf_kappa} not only preserves convexity but, in fact, instantaneously strengthens it.

\begin{thm}[Instantaneous strict convexity when $f=\kappa$]\label{teo:becomes_strict_convex}
 Let $\left(\cM, g \right)$ be a pinched Hadamard surface and  $ \gamma_0 \colon \mathbb{S}^1 \to \cM $ be a smooth,  convex and closed curve. Then the solution $\gamma: \mathbb{S}^1 \times [0, \mathcal{T}) \to \cM$ be  of the flow \eqref{eq:ccf_kappa}, with $h$ as in  \eqref{eq:def_h_LPCF_f} for any $\alpha$, is strictly convex for all $t \in (0, \mathcal{T})$.
\begin{proof}
Lemma \ref{lema:preservationofstrictconvexity} ensures that the  curve remains convex as long as the flow exists.  By periodic extension, we can regard $\kappa: \mathbb{S}^1\times [0,\mathcal{T}) \to \cM$ as a function on $I\times [0,\mathcal{T})$, where $I$ is any open bounded interval so that $[0,2\pi] \subsetneq  I$. Set $D:= I\times (0,\tau)$ for some $\tau\in(0,\mathcal{T})$ so that
\begin{equation}
    \sup_{D}\kappa(p,t) \leq \max_{\overline{I} \times [0,\mathcal{T}]} \kappa(p,t) \; \text{ is uniformly bounded;}
\end{equation}
such $\tau$ exists due to the smoothness and convexity of the initial curve $\gamma_0$, which ensures that the curvature $\kappa$ remains uniformly bounded for a short time by standard regularity results.

Suppose there is an instant $t_0\in (0,\tau)$ where $\kappa$ vanishes, that is, $\kappa(p_0,t_0)=0$ for some $p_0\in \overline{I}$. If no such $t_0$ exists, then the curvature remains strictly positive on $(0, \tau)$. We are interested in the case $p_0\in I$, otherwise we are done.
Using Lemma \ref{evol_eq} (e) with $f=\kappa$, we get
\begin{align}
        \partial_t \kappa -  \partial^2_s \kappa &   \geq \kappa \KGauss -h \kappa^2  \geq \ \left(-b^2      -   (\kappa_{\max})^2 \right)  \kappa =: \kappa \cdot c,
    \end{align}
     with $c$ bounded, and $\kappa \geq 0$. Hence we can apply the strong maximum principle (see \cite[Chapter 2.1,  2.2]{friedman1964}) to conclude that $\kappa \equiv 0$ in $\overline{I}\times\{t_0\}$. This contradicts Lemma \ref{Had_low} (c), thus $\kappa\big|_{(0,\tau)} > 0$.

 Now, for any $\varepsilon > 0$ sufficiently small, we can consider the curve $\gamma_{ \tau - \varepsilon }$ as new initial data. By the short-time existence and uniqueness for  \eqref{eq:ccf_kappa}, there exists a unique smooth solution evolving from this initial curve. Since $\gamma_{\tau - \varepsilon }$ is strictly convex, Lemma~\ref{lema:preservationofstrictconvexity} guarantees that the flow preserves strict convexity for all times $ t \geq \tau - \varepsilon $. By  uniqueness, this solution must coincide with the original $\gamma_t$ on the interval $[\tau - \varepsilon, \mathcal T)$. Therefore, $\gamma_t$ is strictly convex for all $ t \in (0, \mathcal{T}) $.  

\end{proof}
\end{thm}

\section{A priori control on the radial distance on pinched Hadamard surfaces} \label{radius}

Hereafter, we will assume that $\gamma_0$ is strictly convex, which adds no further restriction, as we have already shown that any convex curve becomes instantaneously strictly convex under  \eqref{eq:ccf_kappa}. From now on,  $(\cM,g)$ is a pinched Hadamard surface as in \eqref{pinched} and let $\gamma(\cdot,t)$ be a smooth convex solution of \eqref{eq:ccf_kappa}, with global term $h(t)$ as in \eqref{eq:def_h_LPCF_f} for $\alpha =0,1$, on $[0,\mathcal{T})$.

 As a first step, we deduce a lower bound on the inner radius $\rho_{-}(t)$ (defined as in \eqref{eq:innerradius}) of $\Omega_t$. 

\begin{lemma}        \label{lema:bounds_inner_radius}
    Let $\{\Gamma_t\}_{t\in[0,\mathcal{T})}$ be a smooth solution of  \eqref{eq:ccf_kappa}, with  $h(t)$ as in  \eqref{eq:def_h_LPCF_f},  starting from a smooth, closed and convex curve $\gamma_0$.  Then there is a constant $r_1$ depending  on $\gamma_0$ and $a$ so that
    \begin{equation}
        0< r_1 \leq \rho_{-}(t) \leq L_0/2  \qquad \text{for all } \quad  t\in[0,\mathcal{T}).
    \end{equation}

    \begin{proof}
From Lemma \ref{Had_low} (b) and the monotonicity in Lemma \ref{lemma:monotonicityisoperimetricdeficit}, we obtain
\begin{equation}
    \coth\left(\frac{a \rho_{-}(t)}{2}\right)
    \le \frac{L(t) +\sqrt{\Delta(t)}}{A(t)\, a}
    \le \frac{L_0+\sqrt{\Delta(0)}}{A_0\, a},
    \label{eq:desig_coth_inner_radius}
\end{equation}
where the last inequality holds for both flows by Proposition~\ref{prop:APCSF_L_decreases_LPCF_A_increases}. This  and Lemma \ref{Had_low} (a) imply
\begin{equation}
    0 < aA_0 \le aA(t) \le L(t) \le L_0,
    \label{eq:boundofL(t)}
\end{equation}
which ensures that the right-hand side of \eqref{eq:desig_coth_inner_radius} is strictly greater than 1. Then it makes sense to apply $\coth^{-1}$ to both sides and, since $\coth$ is a decreasing function, we attain 
    \begin{equation}
         \rho_{-}(t)  \geq   \frac{2}{a}  \coth^{-1}\left(  \frac{L_0 + \sqrt{\Delta(0)}}{A_0 \, a} \right) =: r_1 > 0 .
        \label{eq:fita_inferior_inner_radius}
    \end{equation}
    
 For the upper bound, 
    by definition,  $\rho_-(t)$ is the radius of the largest geodesic ball $B(p, \rho_-(t))$ contained entirely in  $\Omega_t$ and tangent to $\Gamma_t = \partial \Omega_t$.  Then $B(p, \rho_-(t))$ cannot contain two points that are farther apart than the diameter of $\Omega_t$; and since the latter is convex, we have
\begin{equation} \label{diamL}
 2 \rho_-(t) \leq  \operatorname{diam}(\Omega_t) = \operatorname{diam}(\Gamma_t) \leq L(\Gamma_t) \leq L_0.
\end{equation}
    \end{proof}
    \end{lemma}
    
At this stage we know that there exists a geodesic ball of radius $r_1$ contained in $\Omega_t$ for each $t\in[0,\mathcal{T})$, but the center of such ball may vary with time. Next, we aim to prove the existence of a geodesic ball with fixed center enclosed by the evolving domain on a controlled time interval.

   \begin{lemma}\label{lema:innerbal}
 For any $t_0\in[0,\mathcal{T})$, let $p_0$ be the center of the inner ball $B(p_0,\rho_0)$ of $\gamma_{t_0}$, where $\rho_0 = \rho_- (t_0)$. Then we can find  some $\tau$ depending only on $a, b$  and the initial curve $\gamma_0$ so that
        \begin{equation}
            B(p_0, \rho_0/2) \subset \Omega_t, \quad \text{for all } t\in[ t_0 ,  \mathfrak T) , \quad \text{with} \quad \mathfrak T:= \min\{\mathcal{T},t_0+\tau\}.
            \label{eq:geodesicballwithfixedcenter}
        \end{equation}

\begin{proof}
 Denote by $\mathrm{r}_{p_0} = r_{p_0} \circ \gamma_t$, where $r_{p_0}$ is the distance from $p_0$.     First, thanks to \eqref{Frenet}, we get 
    \begin{align}
        \ps^2 \mathrm{r}_{p_0} & = \partial_s \avg{{T, \pr}} = -\kappa  \avg{ N, \pr} + \avg{T, \nabla_s \pr} 
        \label{eq:laplacianr2}.
    \end{align}
On the other hand, we can write
     \begin{equation}
    \partial_t \mathrm{r}_{p_0} = (h-\kappa) \langle N, \pr \rangle = \ps^2 \mathrm{r}_{p_0} + h \avg{ \pr, N } -  \avg{T, \nabla_s \pr}.
    \label{eq:evol_of_r}
    \end{equation}

   We set $t_1:= \inf \left\{ t > t_0 \ | \ p_0\notin \Omega_t \right\}$, which implies that
   \begin{equation}
       p_0 \in \Gamma_{t_1}  = \partial \Omega_{t_1} \quad \text{and} \quad \avg{ N, \pr } \geq 0 \quad \text{on} \quad [t_0, \min\{t_1,\mathcal{T}\} ),
       \label{eq:condicionsdeprendret1}
   \end{equation}
    because a convex set is star-shaped with respect to any interior point.
    Hence using the upper bound from Lemma \ref{lema:hessian_inequalities_r} (a) and the fact that $h(t)>0$, we get 
    \begin{align}
        \heatoperatorDOSdim \mathrm{r}_{p_0}  \geq - b \coth(b \mathrm{r}_{p_0}) \qquad \text{on} \quad   \mathbb{S}^1\times [t_0, \min\{t_1,T \}).
        \label{eq:fitainferioroperadoredpr}
    \end{align}
    Next, we denote 
       $\mathrm{r}(t):=\min_{ \mathbb{S}^1}  \mathrm{r}_{p_0}(\cdot, t)$ for any fixed $t\in[t_0, \min\{t_1,\mathcal{T}\})$. By Hamilton's trick (cf. \cite[Lemma 2.1.3]{{mantegazza_LNMCF}}), one has
   \begin{align}
       \mathrm{r}'(t)  \geq  - b \coth(b \mathrm{r}(t)).
   \end{align}
    Thus it follows that
   \begin{equation}
       \mathrm{r}(t) \ \geq \ R(t) \ , \qquad \text{for} \quad t \in [t_0, \min\{t_1,\mathcal{T}\}), 
       \label{eq:r_min_major_que_R}
   \end{equation}
   being $R$ the solution to $R' \ = \ - b \coth(b R)$ with  initial condition $R(0) = \rho_0$, that is,
    \begin{equation}
        \cosh(b R)= e^{-b^2 (t-t_0)}\cosh(b \rho_0).
        \label{eq:sol_edo_R}
    \end{equation}

    From here, since $\cosh$ and $s \mapsto \log \left( \frac{\cosh(b s)}{\cosh(b s/2)}\right)$ are  non-decreasing functions, we attain
     \begin{equation}
        R(t)\geq \rho_0/2 \quad \text{if} \quad t-t_0 \leq \frac{1}{b^2}\ln \left( \frac{\cosh( b \, r_1)}{\cosh(b \, r_1/2)}\right) =:\tau,
        \label{eq:fita_R_tau_t1}
    \end{equation} 
    using also the lower bound $r_1$ for the inner radius from Lemma \ref{lema:bounds_inner_radius}. This and  \eqref{eq:r_min_major_que_R} yield
    \begin{equation} \label{bdd_in}
        \mathrm{r}(t) \geq \rho_0/2 \quad \text{if} \quad t\in [t_0, \min\{t_1,t_0+\tau\}).
    \end{equation}
   To complete the proof, we claim that $t_1 \geq \min\{\mathcal{T},t_0+\tau\}=\mathfrak{T} $. Indeed, assume on the contrary that  $t_1  <  \mathfrak{T} $, which allows to use \eqref{eq:r_min_major_que_R} to conclude that
   \begin{equation}
       r_{p_0}(q, t_1- \eta) \ \geq R(t_1-\eta) \quad \text{ for all } q\in \Gamma_{t_1-\eta} \ \text{and all } \ \eta >0.
   \end{equation}
    Hence $\text{dist}(\Gamma_{t_1-\eta}, p_0) \geq R(t_1 - \eta)$,
   and taking limits as $\eta \to 0$,  \eqref{eq:condicionsdeprendret1} and \eqref{eq:fita_R_tau_t1} imply
   \begin{equation}
       0= \text{dist}(\Gamma_{t_1 }, p_0)   \geq R(t_1 )  \geq \rho_0/2>0,
   \end{equation}
    which is a contradiction. Thus the inequality in \eqref{bdd_in} holds for $t\in [t_0,\mathfrak{T})$, meaning that $\Gamma_t$ encloses a ball of radius $\frac{\rho_0 }{  2}$ and fixed centre $p_0$ for all $t\in [t_0,\mathfrak{T})$.
    \end{proof}  
    \end{lemma}

    \begin{lemma}\label{lema:bounds_rp0_and_upper_outer_radius}
      For all $t\in[ t_0 ,   \frak T)$, with $\frak T$ as in Lemma \ref{lema:innerbal}, and all $x\in \Gamma_t$, we get
        \begin{equation}
          r_1/2 \leq r_{p_0}(x) \leq L_0 =: r_2,
        \end{equation}
    where $L_0$ denotes the length of the initial curve $\gamma_0$  and $r_1$ comes Lemma \ref{lema:bounds_inner_radius}. Therefore,  it holds
    \begin{equation}
    \rho_{+}(t) \leq r_2 \qquad \text{for all } \quad t\in [0,\mathcal{T}).
        \label{eq:upperboundouterradius}
    \end{equation}
        \begin{proof}
    Let $t_0\in[0,\mathcal{T})$ and $B(p_0,\rho_0)$ be the inball of $\Omega_{t_0}$, by Lemma \ref{lema:innerbal}  we know that
        $  B(p_0, \rho_0/2) \subset \Omega_t, $
        for all $t\in[ t_0 ,   \frak T)$. Fix $t\in [ t_0 ,   \frak T)$, since $\Omega_t$ is convex \cite[Theorem 1]{Alexander1975} and $p_0\in \text{int}(\Omega_t)$, for every $q\in \Gamma_t$ there exists a geodesic $\alpha_{q,t}\subseteq \Omega_t$ such that $\alpha_{q,t}$ is the unique minimizer in $\cM$ connecting $q = \alpha_{q,t}(0)$ to $p_0 = \alpha_{q,t}(1)$. 
        
       On a simply connected negatively curved domain we can extend $\alpha_{q,t}$, to give a geodesic $\widetilde{\alpha_{q,t}} :[0,1+\varepsilon]\to \cM$ for some $\varepsilon>0$, which is the unique minimizer in $\cM$ joining $q$ to $\widetilde{q}:= \widetilde{\alpha_{q,t}}(1+\varepsilon) \in \Gamma_t$. Moreover, 
        $\widetilde{\alpha_{q,t}}\big((1,1+\varepsilon)\big)\subset \mathrm{int}(\Omega_t)$. Then Lemma \ref{lema:bounds_inner_radius}, and \eqref{diamL} yield
        \begin{equation}
           r_1 / 2 \leq \rho_0/2 \leq r_{p_0}(q) =  L(\alpha_{q,t}) \leq L(\widetilde{\alpha_{q,t}}) = d(q, \widetilde{q}) \leq \text{diam}(\Gamma_t) \leq L_0.
        \end{equation}

        Finally, by definition $\rho_{+}(t) \leq r_p(x)$ for all $p\in \Omega_t$ and $x\in \Gamma_t$, thus
            $\rho_{+}(t) \leq r_2 .$
        \end{proof}
      
    \end{lemma}

\section{Curvature control under constrained flows} \label{curvature}

\subsection{Bounds on the support function}

Hereafter, for each $t_0\in[0,\mathcal{T})$ let $p_0$ be the center of the inball of $\Omega_{t_0}$. For any $x\in \Gamma_t$,  the support function of $\Gamma_t$ with respect to $p_0$ is defined as 
\begin{equation}
u(x,t):= \sinh{r_{p_0}(x)} \langle \pr, N \rangle 
\label{eq:usupportfunction}
\end{equation}
This notion is borrowed from the hyperbolic space of constant curvature $-1$, but has the handicap that in our setting $\sinh{r_{p_0}(x)} \, \pr$ is not a conformal Killing field, which complicates seriously the analysis hereafter.

\begin{cor}[A priori estimates for the support function]  \label{eq:ufitainferior}
 There exists a constant $c > 0$, depending only on $a$, $b$ and $\gamma_0$, such that, for all $t\in[ t_0 ,  \frak T)$,  the support function  satisfies
\begin{equation}
    0 < 2 c \leq u(x,t) \leq \sinh r_2.
\end{equation}
\end{cor}

\begin{proof}
Fix $t\in [t_0, \frak T)$, let $\varphi=\varphi_t$ be the signed counterclockwise angle between $\pr$ and $N$, then
\begin{align}
   & \pr =  \sin \varphi \ T + \cos \varphi \ N,
\end{align}
with $\varphi\in \left[-\frac{\pi}{2},\frac{\pi}{2}\right]$ because $\Gamma_t$ is convex. In order to to bound $\cos \varphi$, using \eqref{eq:avg_T_N_pr}  we compute
\begin{align}
     \partial_s \cos \varphi  = \avg{ \nabla_s \pr, N} + \avg{\pr, \nabla_s N} 
    =  - \kcirc\avg{\pr, T}\avg{\pr, N} + \kappa \avg{\pr, T} =  \sin\varphi (-\kcirc \cos\varphi + \kappa).
\end{align}
Let $s_0 \in \mathbb{S}^1$ be a point where $ \cos \varphi$ attains its minimum. Evaluating at $s_0$ yields
\begin{align}
    0 = \sin\varphi(s_0) (-\kcirc(s_0) \cos\varphi(s_0) + \kappa(s_0)),
\end{align}
which gives two options: $\sin\varphi(s_0)=0$, and thus $\cos \varphi(s) \geq \cos\varphi(s_0)=1$, or $\cos\varphi(s_0) =\frac{\kappa(s_0)}{\kcirc(s_0)}$.
From here, using the lower bound for $\kappa$ from Lemma \ref{lema:preservationofstrictconvexity}, 
the upper bound for $\kcirc$ from \eqref{eq:fita_curvatura_cercle_geodesic}, 
the fact that $\coth(\cdot)$ is decreasing, 
and the lower bound for $r$ from Lemma \ref{lema:bounds_rp0_and_upper_outer_radius}, 
we obtain 
\begin{align}
    \cos \varphi \geq \frac{\varepsilon_0}{b\coth(br_1/2)}=: \varepsilon_0^\ast >0.
\end{align}
Observe that $\varepsilon_0^\ast$ is smaller than $1$ and that the bound doesn't depend on the fixed $t$. Then,  by Lemma \ref{lema:bounds_rp0_and_upper_outer_radius}, for every $t\in[ t_0 ,  \frak T)$  and $x \in \Gamma_t$ we have 
\begin{equation}
u(x,t) = \sinh r_{p_0}(x) \langle \pr , N \rangle \geq \sinh(r_1/2) \, \varepsilon_0^\ast =:  2c > 0.
\end{equation}

The upper bound for $u$ follows directly from Lemma \ref{lema:bounds_rp0_and_upper_outer_radius}.
\end{proof}

\begin{prop} \label{eq:heat_operator_u}
   The support function $u$ satisfies the following evolution equation:
\begin{align}
    \Box u  &= h \cosh r (1-|\pr^{\tang}|^2) + (h-\kappa) \sinh r \avg{\nabla_{N} \pr, N}  - u |\pr^{\tang}|^2  +  u \kappa^2  - \kappa \sinh r \avg{\nabla_s\pr, T}   \\
& - \cosh r \avg{\nabla_s \pr, T} \avg{N, \pr}  - 2\cosh r \avg{\pr, T} \avg{\nabla_s \pr, N} - 2\kappa \cosh r |\pr^{\tang}|^2 \\
& + \sinh r \left[ \ps \kcirc \avg{T, \pr}\avg{N, \pr} + \kcirc\left(\kappa (2|\pr^{\tang}|^2-1) + \avg{T, \nabla_{s} \pr}\avg{N, \pr} + \avg{T, \pr}\avg{N, \nabla_{s} \pr} \right)\right]. 
\end{align}
\begin{proof} 
Setting $r = r_{p_0}$ for simplicity, let us begin by computing 
\begin{align}
\ps^2 u &= (\ps^2 \sinh r)  \avg{\pr, N} + 2 \ps \sinh r \, \ps \avg{\pr, N} +\sinh r  \, \ps^2 \avg{\pr, N}. \label{eq:nabla_u_1}
\end{align}

We now analyse each term in \eqref{eq:nabla_u_1} separately. Applying \eqref{Frenet}, we get
\begin{align*}
\ps^2 \sinh r &= (\sinh)''(r) (\ps r)^2 + (\sinh)'(r) \ps^2 r \\
&= \sinh r \avg{T, \pr }^2 + \cosh r  \big( \avg{\nabla_s \pr, T } - \kappa \avg{\pr, N}\big).
\end{align*}
 If $\kcirc$ denotes the curvature of the geodesic circle of radius $r$ centered at $p_0$ defined in \eqref{eq:def_kcirc}, thanks to \eqref{eq:avg_T_N_pr} and Frenet formulas, we have
\begin{align*}
\ps^2 \avg{\pr, N} =&  \ps \avg{\nabla_s\pr, N} + \ps \avg{\pr, \nabla_s N} \\
=& - \kcirc\Big[\kappa (2|\pr^{\tang}|^2-1) + \avg{T, \nabla_{s} \pr}\avg{N, \pr} + \avg{T, \pr}\avg{N, \nabla_{s} \pr} \Big]\\
&    -\ps \kcirc \avg{T, \pr}\avg{N, \pr}   + \ps\kappa \avg{\pr, T} + \kappa\avg{\nabla_s \pr, T} -\kappa^2 \avg{\pr,  N},
\end{align*}
which follows from $\avg{\partial_r, T}^2 - \avg{\partial_r, N}^2 = 2|\pr^{\tang}|^2-1$. Now, we can write
\begin{align*}
\avg{\ps \sinh r, \ps \avg{\pr, N}}  
&= \cosh r \avg{\pr, T} \avg{\nabla_s \pr, N} + \kappa \cosh r |\pr^{\tang}|^2.
\end{align*}

On the other hand, the evolution equation of $r$ \eqref{eq:evol_of_r} and Lemma \ref{evol_eq} (b)  ensure
\begin{align*}
\partial_t u 
&= (h-\kappa) \cosh r (1-|\pr^{\tang}|^2) + \sinh r \avg{\nabla_{t} \pr, N} + (\ps \kappa) \sinh r \avg{\pr, T},
\end{align*}
 Combining all the above expressions, we finally reach the formula in the statement.
\end{proof}
\end{prop}

\subsection{Curvature bounds}
Given the constant $c$ from \eqref{eq:ufitainferior}, we consider the auxiliary function 
        \begin{equation}
            {\mathscr W}(x,t)=\frac{ \kappa(x,t) }{u(x,t)-c}.
            \label{eq:auxiliary_W} \qquad \text{for} \quad t\in[t_0, \frak T).
        \end{equation}
  By Corollary \ref{eq:ufitainferior} it suffices to estimate ${\mathscr W}$ from above in order to reach an upper bound for $ \kappa $.

\begin{prop}\label{prop:curvature_bounded}
	Let $(\cM,g)$ be a pinched Hadamard surface and let $\gamma(\cdot,t)$ be a smooth convex solution of the flow \eqref{eq:ccf_kappa}, with global term $h(t)$ as in  \eqref{eq:def_h_LPCF_f}, on $[0,\mathcal{T})$. Then 
	\begin{equation}
	\max_{\Gamma_t}\kappa \leq \overline{\kappa}_0
	\end{equation}
	for any $t\in[0,\mathcal{T})$, where $\overline{\kappa}_0$ depends on $\gamma_0$, $a$ and $b$.
\end{prop}   

\begin{proof}

Using Proposition \ref{eq:heat_operator_u} and Lemma \ref{evol_eq} (e), we can obtain an upper bound for  the heat operator of $\mathscr W$ by discarding all non-positive terms. Indeed, on $[ t_0, \mathfrak T)$, we get
\begin{align}
    \Box \mathscr W & \leq   \,   \frac{ -h\KGauss}{u-c}+\mathscr W^3(u-c)^2 + 2 \ps \mathscr W  \ps u    + \mathscr W^2  \sinh r \avg{\nabla_{N} \pr, N} +  \frac{\mathscr W}{u-c} u |\pr^{\tang}|^2 -  \mathscr W^3 u(u-c)     \\
    &+  \frac{\mathscr W}{u-c} \left[  \cosh r \avg{\nabla_s \pr, T} \avg{N, \pr}  +  2\cosh r \avg{\pr, T} \avg{\nabla_s \pr, N} \right] + 2 \mathscr W^2     \cosh r |\pr^{\tang}|^2 \\
    & + \mathscr W \sinh r \bigg(  \mathscr W \avg{ \nabla_s \pr, T} -  \frac{\ps \kcirc}{u-c}   \avg{T, \pr}\avg{N, \pr} + \mathscr W   \kcirc  -\kcirc  \frac{1}{u-c}  \avg{T, \pr}\avg{N, \nabla_{s} \pr} \bigg).
    \label{eq:primera_desig_deriv_W}
    \end{align}

    Setting  $\varpi(t):=\max_{\mathbb{S}^1} \mathscr W(\cdot, t)$, Corollary  \ref{eq:ufitainferior} allows us to estimate the global term in \eqref{eq:def_h_LPCF_f} as
        \begin{align}
            h(t)  =\frac{1}{\int_{\mathbb{S}^1} \kappa^\alpha \ ds}  \int_{\mathbb{S}^1} \mathscr W (u-c)  \kappa^\alpha  \ ds    \leq \varpi (\sinh{r_2}-c) = c_3 \varpi. \qquad 
            \label{eq:fitahqueconservalongitud}
        \end{align}

        On the other hand, we can obtain an explicit bound for $\kcirc$ by combining  \eqref{eq:fita_curvatura_cercle_geodesic} with Lemma~\ref{lema:bounds_rp0_and_upper_outer_radius}:
\begin{align}
    a \coth \left(a r_2 \right) \leq \kcirc \leq b \coth\left(b \frac{r_1}{2}\right).
    \label{eq:fita_kcirc}
\end{align}

Next, from Lemma~\ref{lema:bounds_rp0_and_upper_outer_radius},  for every  $t \in [t_0, \mathfrak T)$ we have the inclusion
\begin{align} \label{eq:def_donut}
    \Gamma_t \subset \bigcup_{r \in [r_1/2, r_2]} B_r =: \mathcal A,
\end{align}
where $S_r = \partial B_r$  denotes the geodesic circle as in  \eqref{eq:geodesic_circle_centered_p} centered at $p_0$ of radius $r$. Since the annular region $\mathcal A \subset \mathcal M$ is compact and $\ps \kcirc$ is continuous, we achieve a uniform bound
\begin{align} 
    \sup_{\Gamma_t} | \ps \kcirc | \leq \max_{\mathcal A} |\ps \kcirc| =: \mathbf{k_s}.
    \label{eq:fita_ps_kcirc}
\end{align}

To estimate $|\avg{ \nabla_s \pr, N}|$,
from inequality~\eqref{eq:desigualtat_hessiaR_tangentinormal} observe that  the left-hand side can be negative. However, since the term $1 - 2 \langle T, \pr \rangle \langle N, \pr \rangle$
is bounded between $-1$ and $3$, we obtain
\begin{align}
    \left|  a \coth(a r)  \left(1 - 2 \langle T, \pr \rangle \langle N, \pr \rangle \right) -  b \coth(b r) \right|  \leq \max \left\{ A+B, |3A -B| \right\}.
    \label{eq:desigualtats_part_dreta_hessia_r_mixte}
\end{align}
where we have used the notations
\[
A := a \coth(a r) \leq   b \coth(b r)=:B,
\]
as $x\mapsto  x\, \coth(xr)$ is non-decreasing for fixed $r>0$. This implies the following inequalities:
\begin{align*}
    \max\{ A + B, 3A -B\} \leq 2B \leq 3B - A, \quad \text{and} \quad 
    B - 3A \leq B - A \leq 3B - A.
\end{align*}

Therefore, we attain the refined estimate
\begin{align}
    |\avg{ \nabla_s \pr, N}| 
    &\leq \frac{1}{2} \left( 3 B - A\right) \leq \frac{1}{2} \left( 3 b \coth\left(b \tfrac{r_1}{2}\right) - a \coth(a r_2) \right) =: \frac{1}{2} \zeta,
    \label{eq:desigualtat_hessiaR_tangentinormal_modul}
\end{align}
which follows from Lemma \ref{lema:bounds_rp0_and_upper_outer_radius} and the fact that  $r \mapsto c \coth(cr)$ is non-increasing for any $c>0$.

Using the bounds for $h$ \eqref{eq:fitahqueconservalongitud}, for $\kcirc$  \eqref{eq:fita_kcirc} and its first derivative \eqref{eq:fita_ps_kcirc},  those for $u$  in Corollary \ref{eq:ufitainferior} and for $r_{p_0}$ from Lemma~\ref{lema:bounds_rp0_and_upper_outer_radius}, the Hessian estimates in Lemma \ref{lema:hessian_inequalities_r} and \eqref{eq:desigualtat_hessiaR_tangentinormal_modul}, we reach 
\begin{align}
    \varpi' \leq & \, \varpi \big(  -\varpi^2c^2 +  \varpi \theta  +  \theta_1 \big) \leq \varpi \Big(  - \varpi^2 c^2/2   + \Upsilon \Big)\qquad \text{on} \quad [ t_0, \mathfrak T)
    \label{eq:segona_desig_deriv_W}
    \end{align}
    for uniform constants $\theta$, $\Upsilon$ depending on $r_1, r_2, a$ and $b$.
Maximum principle arguments yield
        \begin{align}
            \varpi(t)\leq \max \big\{ \varpi(t_0), \sqrt{2 \Upsilon}/c \big\}.
\end{align}

        Thus,  from the definitions of $\mathscr W$ and $u$, we conclude
        \begin{align}
            \kappa & = \mathscr W  (u-c) \leq (\sinh(r_2)-c) \max  \Big\{ \max_{\mathbb{S}^1}\mathscr W(\cdot, t_0), \sqrt{2 \Upsilon}/c \Big\} 
        \end{align}
        on $[t_0, \mathfrak T)$. As this occurs for any $t_0$ and $\tau$ does not depend on $t_0$ (Lemma \ref{lema:innerbal}), the iterated application of the above inequality from $t_0=0$ leads to the desired uniform bound.
\end{proof}

Now we are in position to establish the sought long time existence result. 
 \begin{proof}[Proof of Theorem \ref{main} {\rm (1)} and {\rm (2)}]
By Theorem \ref{teo:becomes_strict_convex},  $\Gamma_t$ is strictly convex for all $t>0$.
By Proposition \ref{prop:curvature_bounded}, the curvature 
$\kappa$ remains uniformly bounded for all  $t\in[0,\mathcal{T})$. Thus by Theorem~\ref{teo:k_unbdd_or_T_infty} the solution can be extended  infinitely, and hence $\mathcal T = \infty$.
\end{proof}

  \section{AP-CSF on rotationally symmetric  pinched Hadamard surfaces} \label{rotation}
  
  Hereafter we need to focus on a more particular setting so that there is a unique candidate to solution of the isoperimetric problem, which should be our natural limit around which we can study stability of the flow. We will concentrate on the area-preserving case.
  
  \subsection{Rotational symmetry}\label{sec:rotationally-symmetric} 
  $(\cM, g)$ is rotationally symmetric with respect to a point $\mathfrak{p} \in \cM$, referred to as the pole, if  in polar coordinates $(r, u)$ centered at $\mathfrak{p}$ the metric can be written as
  \begin{equation}
  g = dr^2 + \varphi^2(r)\, g_{\mathbb{S}^1},
  \label{eq:metricrotsymmetric}
  \end{equation}
  where $r \in I \subset [0,\infty)$ is the radial distance, $\varphi(r)>0$ and $g_{\mathbb{S}^1}$ denotes the usual metric on $\mathbb{S}^1$.

  On this surface, we define the conformal vector field $X=\varphi(r)\pr$ (see \cite{Montiel1999}). It can be proven that $\nabla_Y X = \varphi'(r) Y$ for any tangent vector $Y$ to $\cM$.
  For the Gauss curvature, it holds
  \begin{equation}
  \KGauss= - \varphi''(r)/\varphi(r).
  \label{eq:Kgaussym}
  \end{equation}
  As our  surfaces are negatively curved  and $\varphi$ is positive, we deduce that $\varphi'$ is  increasing.
  Moreover, arguing as in \cite[Remark 1]{CaMi2}, we have
  \begin{equation}
  \varphi(0) := \lim_{r \to 0} \varphi(r) = 0 \quad \text{and} \quad  \varphi'(0) := \lim_{r \to 0} \varphi'(r) = 1.
   \label{eq:limita0varphideriv}
  \end{equation}
 Thus we conclude that $
  \varphi'(r) > 0$ with $r \in (0, \infty),$  and thus $\varphi$ is increasing.

In this setting, the curves with constant curvature and bounded radial distance to $\frak p$ can be unduloids (periodic graphs over $\mathbb S^1$), nodoids (curves whose tangent vector is vertical at some point) or {\it  geodesic circles} around $\frak p$ (see \cite[Lemma 2.17]{ritoreisoperimetric}). The latter becomes the only option 
 if we also require that $\KGauss$ is decreasing with respect to $r$ (cf. \cite[Lemma 2.18]{ritoreisoperimetric}). Set
\begin{equation}
    \psi:= (\varphi')^2-\varphi\varphi'',
    \label{eq:definicio_psi_varphi}
\end{equation}
which has, up to multiplication by a positive function, the same derivative with respect to $r$ as the Gauss curvature $\KGauss(r)$ (see \cite[\S 2.6.1]{ritoreisoperimetric}).
 Therefore
\begin{equation}
    \psi \text{ decreasing and } \psi(r)<1 \text{ for all } r>0.
    \label{eq:hipotesis_psi}
\end{equation}

\subsection{Reformulation  in terms of the radial distance}

By applying a tangential reparametrization, the solution of {\sc ap-csf}
$\gamma : \mathbb{S}^1\times [0,\mathcal{T}) \longrightarrow \cM$
can be written as
    \begin{align}
         \gamma(u,t)= \exp_{\mathfrak{p}} r(u,t)u = (r(u,t), u)  
          \label{eq:flowuptotgdiffeo}  
    \end{align}
 in normal coordinates with respect to the pole $\mathfrak{p}$, while \eqref{eq:ccf_kappa} is equivalent to 
\begin{align}
     \partial_t r = & \frac{v(r)}{\varphi(r)}(h-\kappa) = \frac{v(r)}{\varphi(r)}   h +  \frac{1}{ v^2(r)}\hat{\Delta}r -  \frac{\varphi'(r)}{\varphi(r)}\left(1 + \frac{(\pu r)^2}{v^2} \right),
     \label{eq:11_flow_r}
\end{align}
where $\hat{\Delta} = \partial^2_u$ denotes the Laplacian for the sphere metric $\mathbb{S}^1$, and $v(r) =\sqrt{ (\pu r)^2 + \varphi(r)^2}$.

To study the stability of geodesic spheres under the flow, we consider curves in a neighbourhood of a geodesic sphere $ \mathcal{S}=\partial B (\mathfrak{p}, \mathfrak{r}) $  centered at the pole $\mathfrak{p}$.
As before, we can parametrize $\mathcal{S}$ by means of a map
$u \mapsto \exp_{\mathfrak{p}}(\mathfrak{r}u).$ We will say that the curve $\Gamma_t$ is $\varepsilon$-close to $\mathcal{S}$ if it admits a parametrization of the form \eqref{eq:flowuptotgdiffeo} such that the radial distance $r = r_{\mathfrak{p}} \circ \gamma$  belongs to
    \begin{equation}
        \mathbb{U}_\varepsilon = \left\{ r : \mathbb{S}^1\to \mathbb{R} \ \mid \|r - \mathfrak{r}\|_{ C^1(\mathbb{S}^1)} < \varepsilon \right\} . 
        \label{eq:u_varepsilon_entorndelradi}
    \end{equation}

    For $\alpha \in (0,1)$, we shall consider functions $r$ in 
    \begin{equation}
    \mathbb{U}_\varepsilon \, \cap h^{2+\alpha}(\mathbb{S}^1),
    \end{equation}
    so that  \eqref{eq:11_flow_r} is well defined. Here the little-Hölder space $h^{2+\alpha}(\mathbb{S}^1)$ is the closure of smooth functions in  $C^{2,\alpha}(\mathbb S^1)$. 
   Interpolation properties of 
    Hölder-type spaces will be used at several points; we refer the reader to 
    \cite{Lunardi1995,DaPratoGrisvard1979} for a detailed exposition.

\subsection{The operator $\mathcal{G}$ and its linearization}\label{section:the_operator_G_linealization} 
 \eqref{eq:11_flow_r} can be rewritten as
\begin{equation}
\frac{dr}{dt} + \mathcal{G}(r) = 0, \quad \text{with} \quad     \mathcal{G}(r) := \frac{v(r)}{\varphi(r)} (\mathcal{K} - \overline{\mathcal{K}})(r).
     \label{eq:22_edo_operador_G}
\end{equation}
Here we have encoded the curvature operator and its average as
\begin{equation}
\mathcal{K}(r) := \displaystyle \frac{-\varphi(r)}{v^3(r)} \hat{\Delta}r + \frac{\varphi'(r)}{v(r)}\left(1 + \frac{(\pu r)^2}{v^2(r)}\right), \quad \text{and} \quad 
 \overline{\mathcal{K}}(r) =  \fint_{\mathbb{S}^1} \mathcal{K}(r)\, \mu(r),
\label{eq:12_expressionKquasilineal}
\end{equation}
where $\fint_{\mathbb{S}^1} = \frac1{\mathfrak{L}(r)} \int_{\mathbb S^1}$, being $\mathfrak{L}(r)= \int_{\mathbb{S}^1} \mu(r)$  
and  the volume element is given by
\begin{equation}
 \mu(r) = \frac{\varphi^2(r)}{v(r)}\, \hat{\mu} + \frac{\partial_u r}{v(r)} \, dr,
\label{eq:19_volume_element}
\end{equation}
where $\hat{\mu}$ denotes the standard volume element on $\mathbb{S}^1$.

We will use the existence of solutions and of
center manifolds for quasilinear Cauchy problems in \cite[Theorem 3.1 and 4.1]{simonettCenterManifoldsQRDS}. To that end, given $U := \mathbb{U}_\varepsilon \cap h^{1+\alpha}(\mathbb{S}^1)$, the operator $\mathcal{G}$  admits a quasilinear decomposition  as
\begin{equation}
    \mathcal{G}(r) = A(r)r + B(r) \quad \text{ for all } r \in U \cap h^{2+\alpha}(\mathbb{S}^1),
    \label{eq:*_quasilinear_decomp_G}
\end{equation}
where $A \in C^\infty(U, \mathcal{L}(h^{2+\alpha}(\mathbb{S}^1), h^\alpha(\mathbb{S}^1)))$ and $B \in C^\infty(U, h^\alpha(\mathbb{S}^1))$ are given by
    \begin{align}
  \hspace*{-0.2cm}  A(r) \cdot = \frac{v(r)}{\varphi(r)} \left(\!P(r) \cdot -  \fint_{\mathbb{S}^1} (P(r) \cdot) \mu(r) \!\right) \!\quad \text{and} \quad B(r) = \frac{v(r)}{\varphi(r)} \left(\! Q(r) -  \fint_{\mathbb{S}^1} Q(r) \mu(r) \!\right).
    \label{eq:AoperadorquasilinealG}
    \end{align}
Here we have used the notations 
	$P(r) = \frac{- \varphi(r)}{v^3(r) } \hat{\Delta}$ and $Q(r)= \frac{\varphi'(r)}{v(r)}\left(1 + \frac{(\pu r)^2}{v^2(r)}\right).$ From these expressions, $P$ and $Q$ depend smoothly upon $r$, and $P(r)r, Q(r) \in h^{\alpha}( \mathbb{S}^1)$.

\begin{lemma}[Linearization of $\mathcal{G}$] \label{eq:23_linearization_G}
   Given $\psi$ as in \eqref{eq:definicio_psi_varphi},  $L := \mathcal{G}_{*\mathfrak{r}} \in \mathcal{L}(h^{2+\alpha}(\mathbb{S}^1), h^\alpha(\mathbb{S}^1))$ satisfies
\begin{equation}
    L\eta = - \frac{1}{\varphi(\mathfrak{r})^2} \left[\hat{\Delta}\eta + \psi(\mathfrak{r})\eta\right] - \frac{1}{|\mathcal{S}|} \int_{\mathcal{S}} \psi(\mathfrak{r)}\eta \ ,    
\end{equation}

    \begin{proof}
    For any $\eta \in h^{2+\alpha}(\mathbb{S}^1)$, we get $v(\mathfrak{r}+t\eta) = \sqrt{\varphi^2(\mathfrak{r}+t\eta) + t^2(\partial_u\eta)^2}$.
    Thus $v(\mathfrak{r}) = \varphi(\mathfrak{r})$ and
    \begin{equation} \label{eq:16}
    \left. \frac{d}{dt} \right|_{t=0} v(\mathfrak{r}+t\eta) = \frac{1}{2v(\mathfrak{r})} 2\varphi(\mathfrak{r})\varphi'(\mathfrak{r})\eta = \varphi'(\mathfrak{r})\eta.  
    \end{equation}
    From here, taking \eqref{eq:12_expressionKquasilineal} into account, we reach
    \begin{align*}
   \mathcal K_{*\mathfrak{r}}(\eta) 
    & = \left. \frac{d}{dt} \right|_{t=0} \mathcal{K}(\mathfrak{r}+t\eta)= -\frac{1}{\varphi(\mathfrak{r})^2} \partial_u^2 \eta + \frac{\varphi''(\mathfrak{r})}{\varphi(\mathfrak{r})} \eta - \left(\frac{\varphi'(\mathfrak{r})}{\varphi(\mathfrak{r})}\right)^2 \eta   =- \frac{1}{\varphi(\mathfrak{r})^2} \left[ \hat{\Delta}\eta + \psi(\mathfrak{r})\eta \right] . 
    \end{align*}
    
   By \eqref{eq:12_expressionKquasilineal} at $r = \mathfrak{r}$ the curvature $\mathcal{K}$ is constant, and thus $(\mathcal{K}-\overline{\mathcal{K}})(\mathfrak{r}) = 0$. This leads to
    \begin{align}
        L\eta &= \frac{v(\mathfrak{r})}{\varphi(\mathfrak{r})} (\mathcal{K}_{*\mathfrak{r}} - \overline{\mathcal{K}}_{*\mathfrak{r}})(\eta)= \ -\frac{1}{\varphi(\mathfrak{r})^2}  (\hat{\Delta}\eta + \psi(\mathfrak{r})\eta)  -  \overline{\mathcal{K}}_{*\mathfrak{r}}(\eta).
        \label{eq:26}
    \end{align}

    So it remains to compute the linearization of the averaged curvature operator. Indeed, using the divergence theorem, we reach
    \begin{align*}
        \overline{\mathcal{K}}_{*\mathfrak{r}}(\eta) &= \frac{1}{\mathfrak{L}(\mathfrak{r})} \left[ \int_{\mathbb{S}^1} \mathcal{K}_{*\mathfrak{r}}(\eta) \mu(\mathfrak{r}) + \int_{\mathbb{S}^1}  (\mathcal{K}(\mathfrak{r}) - \overline{\mathcal{K}}(\mathfrak{r})) \mu_{*\mathfrak{r}}(\eta) \right] = -\fint_{\mathbb{S}^1}  \psi(\mathfrak{r})\eta \mu(\mathfrak{r}),
    \end{align*}
     Now,  by \eqref{eq:19_volume_element} it holds $\mu(\frak r) = \varphi(\mathfrak{r}) \hat{\mu}$ and  $\mathfrak{L}(\mathfrak{r})   = \varphi(\mathfrak{r}) |\mathbb{S}^1| = |\mathcal{S}|,$
        where the last equality follows because a geodesic sphere of radius $\mathfrak{r}$ is isometric to a sphere of radius $\varphi(\mathfrak{r})$ in the Euclidean space. Therefore, we get $\overline{\mathcal K}_{*\mathfrak{r}}(\eta) = \frac{-\varphi(\mathfrak{r})}{|\mathcal{S}|}  \int_{\mathbb{S}^1}  \psi(\mathfrak{r})\eta  \hat{\mu}$, and the statement follows by substitution.
    \end{proof}
    \end{lemma}

Next, we study the spectrum of the linearization of the operator $\mathcal{G}$.
    
\begin{prop}[The spectrum of $ -\mathcal{G}_{*\mathfrak{r}}$]\label{Prop:spectrum_-L}
Assuming \eqref{eq:hipotesis_psi}, the spectrum of $-L$ consists of a sequence of negative real numbers
$ \dots < \lambda_{k+1} < \lambda_k < \lambda_{k-1} < \dots < \lambda_1 < \lambda_0 = 0 $. Moreover,
\begin{equation}
\mathcal{L}_1 :=  \ker(-L) = \{1\},
\end{equation}
where $\{1\}$ denotes the space of constant functions on~$\mathbb{S}^1$. In particular,  $\lambda_0$ has multiplicity one.

\begin{proof}
    Let $\phi_i$  be defined as $\Hat{\Delta}\phi_i = \hat{\lambda}_i \phi_i$ for each $i\in \mathbb N_0$.
    Then, for $i \geq 1$, Lemma \ref{eq:23_linearization_G} yields
     \begin{equation}
        (-L)\phi_i = \frac{1}{\varphi(\mathfrak{r})^2} \left[ \hat{\lambda}_i\phi_i + \psi(\mathfrak{r})\phi_i - \frac{1}{|\mathcal{S}|} \varphi(\mathfrak{r}) \psi(\mathfrak{r}) \int_{\mathbb{S}^1} \phi_i \hat{\mu} \right] = \frac{1}{\varphi(\mathfrak{r})^2} (\hat{\lambda}_i + \psi(\mathfrak{r}))\phi_i,
      \label{eq:28}
    \end{equation}
   because it is well-known that $\int_{\mathbb{S}^1} \phi_i = 0$. In turn, for $i=0$, we have $\hat{\lambda}_0=0$ and the eigenfunctions $\phi_0$ are constant, thus $(-L)\phi_0 = 0$, and we conclude that  $\mathcal{L}_1 := \ker(-L)$ contains $\{1\}$.
    
    For  any eigenfunction $\Phi\in L^2(\mathbb{S}^1)$ of $-L$ with eigenvalue $\zeta$, by linearity and \eqref{eq:28}, we get
    \begin{align*}
      \zeta \Phi =  (-L)\Phi  =    (-L) \sum_{i=0}^{\infty} \Phi^i \phi_i &= \sum_{i=0}^{\infty} \Phi^i (-L)(\phi_i) = \sum_{i=1}^{\infty} \Phi^i \frac{1}{\varphi(\mathfrak{r})^2} (\hat{\lambda}_i + \psi(\mathfrak{r}))\phi_i,
    \end{align*}
  with $\Phi^i = \langle \Phi, \phi_i \rangle_{L^2}$. From here, it also holds
    \begin{equation} 0 = \zeta \Phi^0 \phi_0 + \sum_{i=1}^{\infty} \Phi^i \left( \zeta - \frac{\hat{\lambda}_i + \psi(\mathfrak{r})}{\varphi(\mathfrak{r})^2} \right) \phi_i \end{equation}
    
    As $\{\phi_i\}_{i=0}^{\infty}$ are linearly independent 
    and $\Phi \neq 0$ for being an eigenfunction, there exists $i_0$ with $\Phi^{i_0} \neq 0$. If $i_0=0$, then $\zeta =0$;  $i_0 \geq 1$ gives
    $ \zeta = \frac{\hat{\lambda}_{i_0} + \psi(\mathfrak{r})}{\varphi(\mathfrak{r})^2} \ . $ In short, for $i \geq 1$ $\phi_i$ is an eigenfunction of $\hat{\Delta}$ with eigenvalue $\hat{\lambda}_i = -i^2$ if and only if $ \phi_i$ is an eigenfunction of $(-L)$ with eigenvalue $\lambda_i = \frac{-i^2 + \psi(\mathfrak{r})}{\varphi(\mathfrak{r})^2}$; and $\phi_0$ is an eigenfunction of $(-L)$ with eigenvalue $0$. This leads to the statement taking into account \eqref{eq:hipotesis_psi}.
   \end{proof}
\end{prop}

\subsection{Short time existence and  center manifold}
We aim to verify that \eqref{eq:22_edo_operador_G} satisfies the hypotheses of \cite[Theorem 3.1 and 4.1]{simonettCenterManifoldsQRDS}. In this context, we consider the  spaces $E_1 = h^{2+\alpha}(\mathbb{S}^1)$ and $E_0 = h^\alpha(\mathbb{S}^1)$, and we denote by $E_\alpha = (E_0,E_1)_\alpha$ for $ \alpha\in (0,1)$ the corresponding real interpolation space. For simplicity, we shall often omit  $\mathbb{S}^1$ when denoting these spaces.

 We start with a crucial maximal regularity result (check \cite{Angenent1990} for the precise definition of $\mathcal M_\gamma$):

\color{black}

\begin{lemma}\label{prop:maximal_regularity_operator_A}
For any $r \in \mathbb{U}_\varepsilon \cap h^{1+\alpha}(\mathbb{S}^1)$, the linear operator $A(r) \in \mathcal{L}(h^{2+\alpha}(\mathbb{S}^1), h^\alpha(\mathbb{S}^1))$, which is defined as in \eqref{eq:AoperadorquasilinealG}, satisfies 
\begin{equation}
A(r) \in \mathcal{M}_\gamma(h^{2+\alpha}(\mathbb{S}^1), h^\alpha(\mathbb{S}^1)) \quad \text{for any } \gamma \in (0,1). 
\label{eq:cond_4.60}
\end{equation}
\begin{proof}
Fix any $r \in \mathbb{U}_\varepsilon \cap h^{1+\alpha}$. We can decompose $A(r) = R_1 + R_2$, being
\begin{equation}
R_1 h =  -\frac{1}{v^2(r) \ } \Hat{\Delta}h \quad \text{and} \quad R_2 h = -\frac{v(r)}{\varphi(r)}  \fint_{\mathbb{S}^1} P(r) h \, \mu(r).
\label{eq:sum_decomposition_operator_R_A_4.60}
\end{equation}
For any fixed $\sigma \in (0, \alpha)$, notice that $A(r) \in \mathcal{L}(h^{2+\sigma}(\mathbb{S}^1), h^\sigma(\mathbb{S}^1))$. Now, given  $\xi \in T_p^*\mathbb{S}^1$, for the principal symbol of $R_1$ we have
\begin{align}
s(\cdot, \xi) := (\sigma R_1)_p(\xi) = \frac{1}{v^2(r) }  \xi_1 \xi_1 \leq \frac{1}{\varphi^2(r)} |\xi|^2  \leq \frac{|\xi|^2}{\varphi^2(\mathfrak{r} - \varepsilon)} = C_1( \mathfrak{r}, \varepsilon) |\xi|^2,
\end{align} 
where we have used that, as $r \in \mathbb{U}_\varepsilon$,  $|r(\cdot) - \mathfrak{r}| \leq \|r - \mathfrak{r}\|_{C^1} < \varepsilon$  and $\varphi' > 0$. In turn, we also get
\begin{align*} 
s(\cdot, \xi) & = \frac{1}{\left(\varphi^2(r)+(\pu r)^2 \right)  }  \xi_1 \xi_1  \geq \frac{|\xi|^2}{  \varphi^2(\mathfrak{r}+ \varepsilon )  + \varepsilon^2 } = C_2( \mathfrak{r}, \varepsilon) |\xi|^2, 
\end{align*}
which follows from $|\pu r| = |\pu (r-\mathfrak{r})| \leq \|r-\mathfrak{r}\|_{C^1} < \varepsilon$.
Thus we conclude that  $R_1$ is a uniformly elliptic operator in  $\mathbb{S}^1$. Hence,   
\cite[Theorem 4.2, Remark 4.6, Remark 3.1 (c)]{Am} together with a density argument ensure that $-R_1$ generates a strongly continuous analytic semigroup on $h^\sigma$.

Now, from Corollary 1.2.19 in \cite{Lunardi1995}, together with a density argument, we have
\begin{equation}
h^\alpha = (h^\sigma, h^{2+\sigma})_{(\alpha-\sigma)/2}. 
\label{eq:h_alpha_result_interpolation}
\end{equation}
Using the same notation for $R_1$ as an operator either in $\mathcal{L}(h^{2+\alpha}, h^\alpha)$ or in $\mathcal{L}(h^{2+\sigma}, h^\sigma)$, we have
\begin{equation}
R_1^{-1}(h^\alpha) \cap h^\alpha \cap h^{2+\sigma} = h^{2+\alpha} \cap h^\alpha  \cap h^{2+\sigma} = h^{2+\alpha},
\label{eq:equivalence_h2alpha_R1_spaces}
\end{equation}
where in the last equality we have used the continuous embedding $ h^{2+\alpha} \subset h^{2+\sigma} $.

Note that, using our conclusions about $-R_1$, \eqref{eq:h_alpha_result_interpolation} and \eqref{eq:equivalence_h2alpha_R1_spaces},  the hypotheses of Theorem \cite[Theorem 2.14]{Angenent1990} are satisfied, which yields $R_1 \in \mathcal{M}_\gamma(h^{2+\alpha}, h^\alpha)$ for any $\gamma \in (0,1)$.

On the other hand, by its definition, $R_2$ can be regarded as an operator in $\mathcal{L}(h^{1+\alpha}, h^\alpha)$.  Therefore, applying the perturbation result from Lemma 2.7 (c) in \cite{ClementSimonett2001} to the spaces $E_0 = h^\alpha$, $E_1 = h^{2+\alpha}$ and $E_\gamma = h^{1+\alpha} = (h^\alpha, h^{2+\alpha})_{1/2}$, it follows $R = R_1 + R_2 \in \mathcal{M}_\gamma(h^{2+\alpha}, h^\alpha).$
\end{proof}
\end{lemma}

Now we are in position to prove:
\begin{thm}[Short time existence]\label{teo:short_time_existence_r}
    Let $\beta \in (\alpha, 1)$ be fixed. For every $r_0\in \mathcal{V} := \mathbb{U}_\varepsilon \cap h^{1+\beta}(\mathbb{S}^1)$, there exists a  $T = T(r_0) > 0$ such that \eqref{eq:22_edo_operador_G} has a unique maximal solution
$$r \in C([0, T), \mathcal{V}) \cap C^\infty\big(\mathbb{S}^1 \times (0, T) \big).$$
\begin{proof}
First, observe that  \eqref{eq:h_alpha_result_interpolation} implies, in particular, that
\begin{equation}
    h^{1+\alpha} = (h^\alpha, h^{2+\alpha})_{1/2} \quad \text{and} \quad h^{1+\beta} = (h^\alpha, h^{2+\alpha})_{(1+\beta-\alpha)/2}.
\end{equation}
Moreover, as in \cite[proof of Theorem 2.11.1]{Am2},  we get following the dense and continuous injections:
\begin{equation}
h^{2+\alpha}(\mathbb{S}^1) \stackrel{d}{\hookrightarrow} h^{1+\beta}(\mathbb{S}^1) \stackrel{d}{\hookrightarrow} h^{1+\alpha}(\mathbb{S}^1) \stackrel{d}{\hookrightarrow} h^\alpha(\mathbb{S}^1). 
\label{eq:inclusio_h_short_time_existence}
\end{equation}
The proof consists in verifying that \eqref{eq:22_edo_operador_G} satisfies the hypotheses in \cite[Theorem 3.1]{simonettCenterManifoldsQRDS}, taking
\begin{equation} \label{ansatz}
(E_0, E_1) = (h^{\alpha}, h^{2+\alpha}), \quad \gamma = 1/2 , \quad \mu = (1+\beta-\alpha)/2.
\end{equation}
 Indeed,
 if $A$ and $B$ are as in \eqref{eq:AoperadorquasilinealG}, from \eqref{eq:inclusio_h_short_time_existence}, it holds $\mathcal{V} = h^{1+\beta} \cap \mathbb{U}_\varepsilon \subset h^{1+\alpha} \cap \mathbb{U}_\varepsilon = U$; therefore, Lemma \ref{prop:maximal_regularity_operator_A} implies in particular that
    $     A(u) \in \mathcal{M}_\mu(h^{2+\alpha}, h^\alpha)$ for every  $u \in \mathcal{V}.  $

Thus, given $r_0 \in \mathcal{V}$,  \cite[Theorem 3.1]{simonettCenterManifoldsQRDS}  implies the existence of a constant $T = T(r_0) > 0$ and a unique maximal solution
$r \in C([0, T), \mathcal{V}) \cap C^\infty((0, T), h^{2+\alpha})$
to \eqref{eq:22_edo_operador_G} with  $r(0) = r_0$. Finally, the argument in the proof of  \cite[Theorem 1]{EscherSimonett1997} yields that the solution is smooth.
\end{proof}
\end{thm}

Hereafter we fix $r_0 \in \mathcal{V}$, with $\mathcal{V}$ as in Theorem \ref{teo:short_time_existence_r}.   Adapting the proof of \cite[Lemma 6.1]{EscherSimonett1998_centermanifold}, including the changes explained in \cite{EscherSimonett1998_VPMCF}, v), one can find a closed subspace $h_s^{2+\alpha}(\mathbb{S}^1)$ of $h^{2+\alpha}(\mathbb{S}^1)$ so that
    $h^{2+\alpha}(\mathbb{S}^1)=\mathcal{L}_1 \oplus h_s^{2+\alpha}(\mathbb{S}^1) $
    is a direct topological sum which reduces the operator $L$.

\begin{prop}[Existence of a center manifold]\label{Prop:existence_center_manifold}
Given $k \in \mathbb{N}_0$, there is an open neighbourhood $\mathcal{O}$ of $\mathfrak{r}$ in $\mathcal{L}_1 = \{1\}$ and a map $\Upsilon \in C^k(\mathcal{O}, h_s^{2+\alpha}(\mathbb{S}^1))$ so that the manifold
$\mathcal{M}^c := \text{graph}(\Upsilon) \subset h^{2+\alpha}(\mathbb{S}^1)$ 
      is locally invariant for  \eqref{eq:22_edo_operador_G}, contains all its small global solutions, and $T_{\mathfrak{r}} \mathcal{M}^c = \mathcal{L}_1$.

\begin{proof}
We aim to show that \eqref{eq:22_edo_operador_G} satisfies the hypotheses (H1), (H2), (H3) of \cite[Theorem 4.1]{simonettCenterManifoldsQRDS}, with the ansatz from \eqref{ansatz}. From the proof of Theorem \ref{teo:short_time_existence_r}, we already know that the conditions of \cite[Theorem 3.1]{simonettCenterManifoldsQRDS} hold for \eqref{eq:22_edo_operador_G}.     By \eqref{eq:h_alpha_result_interpolation}, 
$ h^{1+\alpha} = (h^{\sigma}, h^{2+\sigma})_\delta$,  with
$\delta = (1+\alpha-\sigma)/2$. Then from \cite{Am2} (see the proof of Theorem 2.11.1) we have 
\begin{equation}
    \|r\|_{h^{1+\alpha}} \leq C \|r\|_{h^{\sigma}}^{1-\delta} \|r\|_{h^{2+\sigma}}^{\delta} \quad \text{for all } r \in h^{\sigma}(\mathbb{S}^1),
\end{equation} 
which is the first assumption (H1).

On the other hand, using \eqref{eq:u_varepsilon_entorndelradi}, it is clear that $\mathfrak{r} \in \mathbb{U}_\varepsilon$ for all $\varepsilon > 0$, thus $\mathfrak{r} \in U = \mathbb{U}_\varepsilon \cap h^{1+\alpha}.$ In addition, from \eqref{eq:22_edo_operador_G} and \eqref{eq:*_quasilinear_decomp_G}, we have
$ A(\mathfrak{r}) \mathfrak{r} + B(\mathfrak{r})= \mathcal{G}(\mathfrak{r}) = \frac{v(\mathfrak{r})}{\varphi(\mathfrak{r})} (\mathcal{K} - \overline{\mathcal{K}})(\mathfrak{r}) = 0.$
 Hence hypothesis (H2) is also satisfied for $u_0 = \mathfrak{r}$.

 From Proposition~\ref{Prop:spectrum_-L}, and using the notation in Theorem 4.1 in \cite{simonettCenterManifoldsQRDS}, we have
\begin{equation}
    \sigma_s(-L) \subset \{ x \in \R \mid x < 0 \} \quad \text{and} \quad \sigma_c(-L) = \{0\} ;
\end{equation}
being $0$ an eigenvalue of multiplicity one, which leads to (H3). This means that we have verified all the required hypotheses; hence  \cite[Theorem 4.1]{simonettCenterManifoldsQRDS} ensures that
for a small $\delta > 0$  there exists a $C^k$ map
\begin{equation}
    \Upsilon: \mathcal{O}:= B_{\mathcal{L}_1}(\mathfrak{r}, \delta) \longrightarrow h_s^{2+\alpha}(\mathbb{S}^1)
\end{equation}
satisfying the properties of the statement.
\end{proof}
\end{prop}

Next, we aim to prove that, in a small neighbourhood, the center manifold coincides with the
space $\mathcal{E}$ of parametrizations of equilibria for  \eqref{eq:22_edo_operador_G}, that is, with geodesic circles centered at $\mathfrak{p}$.
With this goal, in the neighbourhood of $\mathfrak{r}$ given by $\mathbb{U}_\varepsilon$, $\mathcal{E}$ can
be parametrized as follows:
\begin{equation}
    \begin{array}{cccl}
     \phi: & \R & \longrightarrow & \mathcal{E}\, \cap \mathbb{U}_\varepsilon  \\
          &  s & \longmapsto & \rho_s
\label{eq:def_phi_param_radios}
\end{array}
\end{equation}
such that $\exp_{\mathfrak{p}}[\rho_s(u) u]$ parametrizes a geodesic circle $S_s$ with center $\mathfrak{p}$ and radius $\mathfrak{r}+s$. In normal coordinates with respect to the pole $\mathfrak{p}$, for each $s\in \mathbb R$, we can write $\rho_s(u)= s+\mathfrak{r}$.

Now we are in position to relate the center manifold $\mathcal{M}^c$ to the set $\mathcal{E} \cap \mathbb{U}_\varepsilon$ of geodesic circles in $\cM$ which are small perturbations of the geodesic circle $\mathcal{S} = \partial B(\mathfrak{p}, \mathfrak{r})$.

\begin{prop}\label{prop:center_manifold_coincides_with_stable_solutions}
     There is a neighbourhood $\mathcal{O}' \subset \mathcal{O}$ of $\mathfrak{r}$ in $\mathcal{L}_1$ for which $\mathcal{E}$ coincides
with an open set of the local center manifold $\mathcal{M}^c$ for \eqref{eq:22_edo_operador_G}, that is, $\mathcal{E}$ is the graph of $\Upsilon$ restricted to $\mathcal{O}'$.
\begin{proof}
    By  Proposition \ref{Prop:existence_center_manifold}, we already know that $\mathcal{M}^c$ contains  all small equilibria of \eqref{eq:22_edo_operador_G}. So it remains to show that $\mathcal{M}^c$ and $\mathcal{E}$ coincide near $\mathfrak{r}$. With this aim, first notice that
 at $s=0$ any function $\rho_s \in \mathbb{U}_\varepsilon  \cap \mathcal{E}$ satisfies $\left. \partial_s \rho_s \right|_{0}= 1$,
    hence, applying Proposition \ref{Prop:existence_center_manifold}, we get that $ \left\{ \left.  \partial_s \rho_s \right|_{0} \right\}$ constitutes a basis of $\{1\}=\mathcal{L}_1= T_{\mathfrak{r}}\mathcal{M}^c$.

    On the other hand, the map $\phi$ defined by \eqref{eq:def_phi_param_radios} is a parametrization of $\mathcal{E}\,\cap \mathbb{U}_\varepsilon$, 
    \begin{equation}
        \phi_{* 0} \frac{d}{ds}= \frac{d\rho_s}{ds}
    \end{equation}
   is also a basis of $T_{\mathfrak{r}} \mathcal E$. Then $T_{\mathfrak{r}}\mathcal{E}=T_{\mathfrak{r}}\mathcal{M}^c$ and, since $\mathcal{M}^c$ contains a neighbourhood of $\mathfrak{r}$ in $\mathcal{E}$ they
must coincide in that neighbourhood. The set $\mathcal{O}'$ is the projection of the latter on $\mathcal{L}_1$.   
\end{proof}
\end{prop}

\subsection{Long time existence and exponential attraction}
The following result, which follows from  \cite[Theorem 6.5 and Proposition 6.6]{EscherSimonett1998_centermanifold}, guarantees the long-time existence of solutions to \eqref{eq:22_edo_operador_G} that start sufficiently close to a constant.

\begin{thm}\label{teo:long_time_existence_r}
    Let $k\in \mathbb N $, a constant $\mathfrak{r} >0$ and $\omega \in (0, -\lambda_1)$ be given, with $\lambda_1$ as in Proposition \ref{Prop:spectrum_-L}. There exists a neighbourhood $W=W(k,\omega)$ of $\mathfrak{r}$ in $h^{1+\beta}(\mathbb{S}^1)$ such that, for every  $r_0\in W$,
    \begin{enumerate}
        \item [(a)] the solution $r(\cdot, r_0)$ of  \eqref{eq:22_edo_operador_G} exist on $[0,\infty)$;
        \item [(b)] there exists a unique $\widetilde{\rho}=\widetilde{\rho}(r_0)=(z_0, \Upsilon(z_0))$, with $z_0\in \mathcal{O}$ such that
        \begin{equation}
            \| r(t,r_0)-\widetilde{\rho}\|_{C^k(\mathbb{S}^1)} \leq C(\omega, k, r_0)  e^{-\omega t}, \qquad \text{for all $t>T =T(\omega, k)$.}
        \end{equation}
    \end{enumerate}
\end{thm}

    From Proposition \ref{prop:center_manifold_coincides_with_stable_solutions}, as $\widetilde{\rho}=(z_0, \Upsilon(z_0))\in \mathcal{M}^c$ with $z_0\in \mathcal O$, $\widetilde{\rho}$ parametrizes a geodesic sphere of the Hadamard surface $\mathcal M$. Accordingly, the solution $r(u, t)$ converges  to a geodesic sphere of $\mathcal M$ exponentially fast in the $C^k$-topology, for any $k\in\mathbb N$, as $t$ tends to infinity. Now the proof of Theorem \ref{corm} follows as in \cite{CabezasRivasMiquel2007}.
 
\subsection{Exponential convergence for the AP-CSF of convex curves}

We are now in position to deduce that the solution converges exponentially to a geodesic circle of $\cM$. Indeed,

\begin{itemize}
    \item[\tiny{$\bullet$}]  Proposition \ref{prop:curvature_bounded} yields curvature bounds that imply the long-time existence ($\mathcal T = \infty$) of the flow for any initial convex curve.
    
    \item[\tiny{$\bullet$}] A uniform bound on the radial distance would give a convergent subsequence of curves.
    
    \item[\tiny{$\bullet$}] The fact that the geodesic circle of radius  $\mathfrak{r}$ is an attractor implies that any subsequence converging near $\mathfrak{r}$ must converge to it.
\end{itemize}

We now proceed to establish the second item: the existence of a convergent subsequence under a radial distance bound. 
We first show that if there is no such a uniform bound, then the entire curve must escape to infinity.

\begin{lemma}\label{lema:alternativa_bdd_or_escape}
	Under the hypotheses of Theorem \ref{main} {\rm (4)},  exactly one of the following two alternatives holds:
	\begin{enumerate}
		\item[{\rm (a)}] The family $\{r_t := r_{\frak p} \circ \gamma_t\}_{t \in [0,\mathcal{T})}$ is uniformly bounded, i.e. $\sup_{t \in [0,T)} \max_{p \in \mathbb{S}^1} r_t(p) \leq K$.
		\item[{\rm (b)}] For any $m \in \mathbb{R}^+$, there exists $t \in [0,\mathcal{T})$ such that $\min_{p \in \mathbb{S}^1} r_t(p) > m$.
	\end{enumerate}
\end{lemma}

\begin{proof}
Suppose that (b) does not hold. Then, there is a constant $r_c > 0$ such that
\begin{equation}
\min_{p \in \mathbb{S}^1} r_t(p) < r_c \qquad \text{for all } \quad t \in [0,\mathcal T).
\end{equation}
Then, for each $t$, pick $p_t \in \mathbb{S}^1$ such that $r_t(p_t) < r_c$. By Lemma \ref{lema:bounds_rp0_and_upper_outer_radius}, for any $q \in \mathbb{S}^1$, we have
\begin{equation}
d\big(\gamma_t(q), \gamma_t(p_t)\big) \leq 2\rho_+(t) \leq 2L_0.
\end{equation}
Hence, the radial distance to any point on the curve $\Gamma_t$ satisfies
\begin{equation}
r_t(q) = d(\gamma_t(q), \mathfrak{p}) \leq d(\gamma_t(q), \gamma_t(p_t)) + d(\gamma_t(p_t), \mathfrak{p}) < 2L_0 + r_c.
\end{equation}
That is, $ r_t $ is uniformly bounded by $ 2L_0 + r_c $.
\end{proof}

Accordingly, we can finish the proof of Theorem \ref{main} (4) by means of the following result.
\begin{cor}
	With the same hypotheses as in Lemma \ref{lema:alternativa_bdd_or_escape}, exactly one of the following two alternatives holds:
	\begin{enumerate}
		\item[{\rm (a)}] The family $\{\Gamma_t\}_{t \in [0,\infty)}$ lies on a compact region (containing the pole $\frak p$). 
		\item[{\rm (b)}] There is an increasing subsequence of times $\{t_n\}_{n\geq 0}$ such that 
	$\min_{\mathbb{S}^1} r_{t_n} > n.$
	\end{enumerate}
If {\rm (a)} happens, then the solution  converges exponentially fast in the $C^\infty$-topology to a geodesic circle centered at $\frak p$ enclosing the same area as the initial curve.
\end{cor}
\begin{proof}
	We argue by contradiction. Suppose that the family $\{\Gamma_t\}_{t\in[0,\infty)}$ does not lie on any compact set containing the pole.  
	Then, for every $m\in\mathbb{N}$, there exists a time $t_m\ge 0$ such that \[\min_{\mathbb{S}^1} r_{t_m} > m.\]
	Indeed, this is a consequence of the fact that the curves lie in a uniformly bounded annulus, as we now show.
	From Lemma \ref{lema:bounds_rp0_and_upper_outer_radius}, for each $t_0\in[0,\infty]$ there exists a point $p_0\in \mathcal{M}$ such that	
	\[
	\Gamma_t \subset B(p_0, r_2) \qquad  \text{for all } t \in [t_0,t_0+\tau),
	\]
	where $\tau$ is the constant from Lemma \ref{lema:innerbal}. 
	For any $p,q\in \mathbb{S}^1$, the triangle inequality gives
	\[
	r_t(p)   \leq r_t(q) + d(\gamma_t(p),\gamma_t(q)).
	\]
	Hence,
	\[
	|r_t(p)-r_t(q)| \leq d(\gamma_t(p),\gamma_t(q))
	\quad \text{for all } p,q\in \mathbb{S}^1.
	\]
	Taking the supremum over all $p,q \in \mathbb{S}^1$, we conclude that
	\[
	\max_{p\in \mathbb{S}^1} r_t(p) - \min_{p\in \mathbb{S}^1} r_t(p) \leq \operatorname{diam}(\Gamma_t) \leq 2 r_2.
	\]
	From the assumption, for each $m \in \mathbb{N}$, there exists a time  $t_m\ge 0$ such that
	\[ \Gamma_{t_m} \nsubseteq B(\mathfrak{p}, m+2r_2)\]
	consequently
	\[
	\max_{p\in \mathbb{S}^1} r_t(p) > m+2r_2
	\]
	from where the statement follows.

It remains to show that the times $t_m$ can be chosen to form an increasing sequence. With this aim, 
	assume by contradiction that this is not the case. Then there exists $m_0 \in \mathbb{N}$
	such that
	\[
	t_m \leq t_{m_0} \qquad \text{for all } m \ge m_0.
	\]
	
	In particular, for every $t > t_{m_0}$ we must have
	\[
	\min_{p \in \mathbb{S}^1} r_t(p) \le m_0.
	\]
	Using the uniform bound on the diameter of $\gamma_t$ from Lemma \ref{lema:bounds_rp0_and_upper_outer_radius}, we obtain
	\[
	\max_{p \in \mathbb{S}^1} r_t(p) \le m_0 + 2r_2,
	\]
	and hence
	\[
	\Gamma_t \subset \overline{B(\mathfrak{p},\, m_0 + 2r_2)} \qquad
	\text{for all } t > t_{m_0}.
	\]
	Therefore, the family $\{\Gamma_t\}_{t > t_{m_0}}$ is contained in a compact set.
	
	Since $t_{m_0} < \infty$, we can set $N:=\left\lceil \frac{t_{m_0}}{\tau} \right\rceil<\infty$, where $\tau$ is again the constant from Lemma \ref{lema:innerbal}, and we may cover the interval $[0,t_{m_0}]$ by finitely many
	subintervals of the form
	\[
	I_k := [k\tau,(k+1)\tau), \qquad
	k = 0,\dots,N.
	\]
	Again from Lemma \ref{lema:bounds_rp0_and_upper_outer_radius}, for each such interval $I_k$, there exists a point $p_k$ (chosen at time $k\tau$) 
	such that
	\[
	\Gamma_t \subset \overline{\mathcal A_{p_k}(r_1/2,r_2)} \qquad \text{for all } t \in I_k,
	\]
	with the notation coming from \eqref{eq:def_donut},
	Consequently, we reach
	\[
	\bigcup_{t \in [0,t_{m_0}]} \Gamma_t
	\subset \bigcup_{k=1}^N \overline{\mathcal A_{p_k}(r_1,r_2)}.
	\]
	Altogether, the family $\{\Gamma_t\}_{t\in[0,\infty)}$ is contained in a compact
	set containing the pole, yielding a contradiction. 
		
	To prove the last claim of the statement, assume  that (a) happens, i.e., the radial function $r_t$ is uniformly bounded. From Proposition~\ref{prop:curvature_bounded}, we know that under the flow the curvature $\kappa$ remains uniformly bounded. Accordingly, Theorem~\ref{teo:k_unbdd_or_T_infty} ensures that the solution exists for all times $t \in [0, \infty)$.

	As  the image of the evolving curve $\Gamma_t$ remains in a fixed compact region of $\cM$, since the Gauss curvature $\KGauss$ and all its covariant derivatives are smooth functions on $\cM$, there exists a constant ${C} > 0$ such that
	\[
	\max_{0\leq \ell \leq 2} \, \sup_{\Gamma_t} |\nabla^\ell \KGauss| \leq C
	\]
	for every $t \geq 0$. Therefore, the hypotheses of Theorem \ref{teo:limit_kappa_es_h} are satisfied, which yields
	\[
	\lim_{t \to \infty} \sup_{\mathbb{S}^1} |\kappa(\cdot,t) - h_\infty| = 0.
	\]
	Indeed, notice that the family $\{h(t)\}_{t<\infty}$ is bounded in $\mathbb{R}$, and thus, by the Bolzano--Weierstrass theorem, there exists a sequence $t_n \to \infty$ such that $h(t_n) \to h_\infty$ for some $h_\infty \in \mathbb{R}$. Thus  $\kappa(\cdot,t_n)$ converges uniformly to the constant  $h_\infty$ on $\mathbb{S}^1$.
	
	We know that under our hypotheses the only closed curves with constant curvature are geodesic circles centered at the pole. From the previous discussion, we have that $\gamma(\cdot,t_n)$ converges uniformly to such a geodesic circle. By exploiting higher derivative estimates as those in Proposition \ref{Prop:bounds_higher_deriv_kappa_APCSF_LPCF}, it follows by standard methods that $\gamma(\cdot,t_n)$ is arbitrarily close in $C^\infty$ to a geodesic circle for large $n$.
	
	Therefore, if we took $\gamma(\cdot,t_n)$ as initial data, we would be in the setting of Theorem~\ref{teo:long_time_existence_r}, and the flow would converge exponentially fast to the corresponding geodesic circle in the $C^k$-topology for every $k \in \mathbb{N}$. Since the solution is unique up to reparametrization, the original flow must also converge exponentially (up to reparametrization) to the same geodesic circle.
\end{proof}

Finally, we manage to find a geometric condition that guarantees that  the evolving curve does not escape to infinity. Indeed, condition \eqref{eq:condicio_area} provides a geometric criterion ensuring the boundedness of the radial distance to the pole. 
	It requires that the initial enclosed area $A_0$ is smaller than the area of a geodesic disk of radius $r_1$ in the model space of constant curvature $\KGauss = -b^2$. 
	Here, $r_1$ is a lower bound for the inner radius of the domain enclosed by the curve, and it depends explicitly on the initial geometry through $A_0$ and $L_0$. 
	Although this makes \eqref{eq:condicio_area} an implicit condition, it assures that the flow cannot escape to infinity while preserving area. By Bishop-Gunther comparison theorem, the condition of the area is satisfied for any geodesic circle of radius $r_1$, and hence for any curve in the region enclosed by it.

\begin{prop}\label{lemma:condicio_area_distancia_radial_fitada}
	Under the assumptions of Theorem \ref{main} (4), if the initial  curve $\gamma_0$ satisfies
	\begin{equation}\label{eq:condicio_area}
	A_0 \leq  A_{r_1}^{-b^2}=\frac{2\pi}{b^2} \left(\cosh\left( \frac{2b}{a} \ \coth^{-1} \left( \frac{\sqrt{\Delta(0)} + L_0}{A_0 a} \right)\right)-1 \right) ,
	\end{equation}
	where $r_1$ is the lower bound for the inner radius from Lemma \ref{lema:bounds_inner_radius}.  Then the solution  converges in $C^\infty$ exponentially fast to a geodesic circle centered at $\frak p$ enclosing area $A_0$.

\end{prop}

\begin{proof}
Suppose, by contradiction, that the family $\{r_t\}$ is not uniformly bounded. Then, by Lemma~\ref{lema:alternativa_bdd_or_escape}, for any $m \in \mathbb{R}^+$, there exists $t \in [0,\mathcal T)$ such that $\min_{\mathbb{S}^1} r_t > m$.

Since $\cM$ is a rotationally symmetric surface and $\KGauss$ is \emph{strictly} decreasing, equation \eqref{eq:Kgaussym} implies that the map $r  \mapsto  \frac{\varphi''(r)}{\varphi(r)}$
is a diffeomorphism onto its image. In particular, it admits a smooth inverse function $\Psi : [a^2, b^2) \to \mathbb{R}^+ \cup \{0\},$ which allows us to express the radial coordinate $r$ in terms of the Gaussian curvature. Let us denote $\Psi(c^2) = r_c$.

Then, under the assumption that $r_t$ is not uniformly bounded, for any $r_c > 0$ there exists $t \in [0,\mathcal T)$ such that
\begin{equation}
\min_{p \in \mathbb{S}^1} r_t(p) > r_c, \quad \text{and thus } \quad \max_{p \in \mathbb{S}^1} \KGauss\big(\gamma_t(p)\big) < -c^2.
\end{equation}

    By Lemma \ref{lema:bounds_inner_radius}, for all time $t \in [0, \mathcal T)$ there exists some point $p_t \in M$ such that
    \begin{equation}
        B(p_t, \rho_{-}) \subset \Omega_t,
    \end{equation}
    where $\Omega_t$ denotes the region enclosed by $\Gamma_t$. Then, as the flow preserves area, we have
    \begin{equation}
        A^{-c^2}_{\rho_{-}} \leq A\big(B(p_t, \rho_{-})\big) < A(\Omega_t) = A_0,
        \label{eq:ineq_r_area_c}
    \end{equation}
    where the first inequality follows from \cite[corollary of Lemma 7]{ossermanIospIneq}, and $A^{-c^2}_{\rho_{-}}$ denotes the area of a geodesic disk of radius $\rho_{-}$ in the complete simply-connected surface of constant curvature $-c^2$, which satisfies
    $$A_{\rho_-}^{-c^2}= \frac{4\pi}{c}  \sinh^2\left(\frac{c \rho_{-}}{2}\right) = 2\pi \frac{\cosh(c\rho_{-})-1}{c^2}.$$
    On the other hand, using that the inequality \eqref{eq:ineq_r_area_c} is satisfied for any $c\in [a^2,b^2)$, that the function $\sinh^2\left(\frac{b \, \cdot}{2}\right)$ is non-decreasing, and the lower bound for $\rho_{-}$ from Lemma \ref{lema:bounds_inner_radius}, we attain
    \begin{equation}
           A_{r_1}^{-b^2}=\frac{2\pi}{b^2} \left(\cosh\left( \frac{2b}{a} \ \coth^{-1} \left( \frac{\sqrt{\Delta(0)} + L_0}{A_0 a} \right)\right)-1 \right) < A_0,
          \label{eq:ineq_r_area_b}
    \end{equation}
which contradicts \eqref{eq:condicio_area}. Thus the family $\{r_t\}_{t \in [0,\mathcal T)}$ must be uniformly bounded.
 
\end{proof}

\section*{Acknowledgements} Both authors would like to thank Vicente Miquel for enlightening discussions and continuous support.

\section*{Data availability} No data was used for the research described in the article.


\end{document}